\documentclass[12pt]{article}
\usepackage{enumerate,amsmath,amsthm,latexsym,amssymb} %,hyperref}

\usepackage{graphicx}
\usepackage[usenames,dvipsnames]{xcolor}
\usepackage[margin=1in]{geometry}
\usepackage{tikz}
\usepackage{float}
\usepackage{bbm}
\usepackage{relsize}

\def\bX{{\bf X}}

\def\cL{{\mathcal L}}

\newcommand{\set}[1]{\left\{#1\right\}}
\def\cD{{\mathcal D}}
\def\bo{{\bf 0}}
\def\ux{{\ul x}}
\def\uy{{\ul y}}
\def\uu{{\ul u}}
\def\uv{{\ul v}}
\def\uz{{\ul z}}
\def\cH{{\mathcal H}}

\def\cL{{\mathcal L}}

\newcommand{\brac}[1]{\left( #1 \right)}
\newcommand\bfrac[2]{\left(\frac{#1}{#2}\right)}
\def\E{{\bf E\,}}

\def\Pr{\mathbb{P}}
\def\rai{\rightarrow \infty}
\newcommand{\ol}[1]{\overline{#1}}
\newcommand{\ul}[1]{\boldsymbol{#1}}
\def\mod{\mbox{\rm mod}\;}
\def\rank{\mbox{\rm rank}}
\def\corank{\mbox{\rm co-rank}}

\def\nul{\mbox{\rm null}}
\def\ooi{(1+o(1))}
\newcommand{\ind}[1]{\,\mathbbm{1}_{\{{#1}\}}}
\newcommand{\indi}[1]{\,{\mbox{\large$\mathbbm{1}$}\{{#1}\}}}

\newcommand{\ignore}[1]{ }
\def\cS{{\mathcal S}}
\def\cF{{\mathcal F}}
\def\a{\alpha}
\def\b{\beta}
\def\d{\delta}
\def\D{\Delta}
\def\e{\varepsilon}
\def\f{\phi}
\def\F{\Phi}
\def\g{\gamma}

\def\k{\kappa}

\def\Th{\Theta}
\def\l{\lambda}

\def\p{\pi}

\def\r{\rho}

\def\s{\sigma}
\def\S{\Sigma}

\def\om{\omega}
\def\Om{\Omega}

\def\cB{{\mathcal B}}

\newtheorem{theorem}{Theorem}

\newtheorem{lemma}[theorem]{Lemma}

\newtheorem{remark}[theorem]{Remark}%

\newcommand{\qchoose}[3]{\mbox{$%
\left[\begin{array}{c}#1\cr#2\end{array}\right]_{#3}$}}

\newcommand{\ra}{\rightarrow}

\newcommand{\beq}[2]{\begin{equation}\label{#1}#2\end{equation}}
\newcommand{\mult}[2]{\begin{multline}\label{#1}#2\end{multline}}

\def\sm{\! \setminus \!}
\def\sd{\! \oplus \!}
\def \ssm{\setminus}

\newcounter{rot}%\addtocounter{rot}{1}, \therot

\def\es{\emptyset}

\def\nn{\nonumber}
\def\wh{\widehat}

\begin{document}

\title{Rank of the vertex-edge incidence matrix of $r$-out hypergraphs}

\author{Colin Cooper\\Department of Informatics\\
King's College\\
London WC2B 4BG\\England
\and Alan Frieze\thanks{Research supported in part by NSF Grant DMS1661063}\\
Department of Mathematical Sciences\\
Carnegie Mellon University\\
Pittsburgh PA15213\\
U.S.A.}
\maketitle

\begin{abstract}
We consider a space of sparse Boolean matrices of size $n \times n$, which
have finite co-rank over $GF(2)$ with high probability. In particular, the probability that such a matrix  has full rank, and is thus invertible, is a positive constant with value about $0.2574$ for large $n$.

The matrices arise as the vertex-edge incidence matrix of 1-out 3-uniform hypergraphs.
The result that the null space is finite, can be contrasted with results for the usual models of sparse Boolean matrices, based on the vertex-edge incidence matrix of random $k$-uniform hypergraphs. For this latter model, the expected co-rank  is linear in the number of vertices $n$, \cite{ACO}, \cite{CFP}.

For fields of higher order, the co-rank is typically Poisson distributed.
\end{abstract}
\section{Introduction}

For   positive integers $r\geq 1,\,s\geq 2$, let $\ul M(s,r,n)$ be the space of $n\times rn$ matrices with entries generated in the following manner. For each $i=1,...,n$ there are $r$ columns $C_{i,j},\;j=1,...,r$. Each  column $ C_{i,j}$ has a unit entry in row $i$, and  $s\!-\!1$ other unit entries, in rows chosen randomly with replacement from $[n]$, or without replacement from $[n]-\{i\}$, all other entries in the column being zero. In general we consider the arithmetic on entries in the matrix, (and thus the evaluation of linear dependencies),  to be over $GF(2)$. If so, in the ``with replacement case'', if two unit entries coincide  the entry is set to zero. When $r=1$, the matrix consists of an identity matrix plus $s\!-\!1$ random units in each column.

If $s=2$, and entries are chosen without replacement, $M$ is the vertex-edge incidence matrix of the random graph $G_{r-\text{out}}(n)$. This model of random graphs has been extensively studied, and  is known to be $r$-connected for $r\geq2$, Fenner and Frieze \cite{FF}, to have a perfect matching for $r\geq2$, Frieze \cite{F}, and to be Hamiltonian for $r\geq 3$, Bohman and Frieze \cite{BF}. If $s\geq3$ we are considering $r$-out, $s$-uniform hypergraphs.
Random Boolean matrices based on the vertex-edge incidence matrix of $s$-uniform hypergraphs where the columns (edges) are chosen i.i.d. from all columns with $s$ ones  were studied by Cooper, Frieze and Pegden, \cite{CFP}.
A very general paper by Coja-Oghlan, Erg\"{u}r, Gao, Hetterich and Rolvien, \cite{ACO}, gives the limiting rank in this latter model for a wide range of assumptions on the distribution of non-zero entries in the rows and columns. The fundamental difference between the $r$-out model of random matrices, and those of \cite{ACO}, \cite{CFP}  is the presence of an $n \times n$ identity matrix as a sub-matrix (in the without replacement case).

A set of vectors is said to be linearly dependent if there is a nontrivial linear combination of the vectors that equals the zero vector. If no such linear combination exists, then the vectors are said to be linearly independent.
The (row) rank $\r$ of an $n \times m$ matrix, $(m \ge n)$ is the maximum number of linearly independent rows, and the co-rank is $n-\r$.
If the field is $GF(2)$,
${\ul x}\in \set{0,1}^n$ is a linear dependency ({\em dependency} for short) if ${\ul x}M=0$. Let $|{\ul x}|=|\set{j:x_j=1}|$. We say that a set of rows $D\subseteq [n]$ is a dependency if $D=\set{j:x_j=1}$ for some dependency ${\ul x}$. An $\ell$-dependency is one where  $|{\ul x}|=\ell$ or $|D|=\ell$.

Of particular interest is the case $r=1$ which gives $n \times n$ Boolean matrices.
We will show that over $GF(2)$, for $r=1, s=3$,  the linear dependencies among the rows of $M$  are w.h.p. either small (bounded in expectation) or large (of size about $n/2$), and  the distributions of these dependencies are somewhat entangled. For $r=1, s=3$, define a Poisson parameter $\f$ for small dependencies. The value of $\f$  differs marginally in summation range between the ``with replacement'' $\f_R$, and ``without replacement'' models $\f_{\ol R}$ as follows:
\begin{equation}\label{phiNR}
\f_R=
\sum_{\ell \ge 1} \frac 1\ell (2e^{-2})^\ell \sum_{j=0}^{\ell-1} \frac{\ell^j}{j!}, \qquad \qquad \f_{\ol R}=
\sum_{\ell \ge 2} \frac 1\ell (2e^{-2})^\ell \sum_{j=0}^{\ell-2} \frac{\ell^j}{j!}.
\end{equation}
The numeric values are  $\f_R\approx 0.5215$, and  $\f_{\ol R} \approx 0.1151$, where $a \approx b$  means approximately equal.

Let $\pi$ be the probability distribution  given by
\begin{equation}
\label{ldef}
\pi(k) = \begin{cases}\;\;\prod_{j=1}^{\infty} \brac{1- \bfrac{1}{2}^j}&k=0.
\\ \;\;
\frac{\prod_{j=k+1}^{\infty} \brac{1-\bfrac{1}{2}^j}}{\prod_{j=1}^{k} \brac{1-\bfrac{1}{2}^j}} \bfrac{1}{2}^{k^2}&k \ge 1.
\end{cases}
\end{equation}
For $0 \le r\le m$ let
\begin{equation}\label{JKM0}
P^*(h,h+r;m)= \qchoose{m}{r}{2}\;\bfrac{1}{2}^{(h+r)(m-r)} \prod_{j=h+1}^{h+r} \brac{1-\bfrac{1}{2}^j},
\end{equation}
where empty products are treated as unity, and $  {\scriptstyle \qchoose{m}{r}{q}}=\frac{(q^m-1)...(q^{m-r+1})}{(q-1)...(q^r-1)}$.
Let
\begin{flalign}\label{Pmj}
P(\s,\l)&=\frac{\f^\s}{\s!}\,e^{-\f} \;
\sum_{r=0}^\s \pi(\l+r) P^*(\l,\l+r,\s),
\end{flalign}
where here and later in the paper, $\s$ indicates the dimension of the space induced by {\em small} dependencies and $\l$ indicates the dimension of the space induced by {\em large} dependencies.
\begin{theorem}\label{TH1}
Let $M$ be chosen u.a.r. from $ \ul M(3,1,n)$. Let $d \ge 0$ be integer.
The limiting probability that, over $GF(2)$, the matrix $M$ has co-rank $d$,
  is given by
\beq{result}{
\lim_{n\to\infty}\Pr(\corank(M)=d)= \sum_{\s=0}^d P(\s,d-\s).
}
In particular,
\[
\Pr(\rank(M)=n) \sim P(0,0)= e^{-\f}\p(0)=e^{-\f}\prod_{j=1}^\infty  \brac{1-\bfrac{1}{2}^j}.
\]
\end{theorem}
Theorem \ref{TH1} differs from many previous results on sparse random Boolean matrices. The co-rank (dimension of the null space) is finite, and the matrix is invertible with probability $e^{-\f}\pi(0)$, where $\pi(0) \approx 0.2888$. The problem  can be seen as an
instance of the change in rank, if any, arising from small perturbations of the identity matrix.

The finite co-rank given in Theorem \ref{TH1} can be contrasted with results for the edge-vertex incidence matrix of random hypergraphs,
(\cite{ACO}, \cite{CFP}), where the expected co-rank is linear in  the number of vertices $n$, and the probability of a full rank matrix is exponentially small.

The joint distribution of co-rank given by \eqref{Pmj} is a   mixture of a Poisson with parameter $\f$ given in \eqref{phiNR}, and the distribution $\pi$ given in \eqref{ldef}.
This mixture arises due to a gap property in the size of the dependencies (small or large), which we next explain.
%The following theorem summarises the size of dependency we can expect:

\ignore{
Let ${\ul x}\in \set{0,1}^n$ be a {\em dependency} if ${\ul x}M=0$. Let $|{\ul x}|=|\set{j:x_j=1}|$. We say that a set of rows $D\subseteq [n]$ is a dependency if $D=\set{j:x_j=1}$ for some dependency ${\ul x}$. An $\ell$-dependency is one where  $|{\ul x}|=\ell$ or $|D|=\ell$.
}

\begin{theorem}\label{th1}
Let $M$ be chosen u.a.r. from $ \ul M(3,1,n)$, then
w.h.p. either (i) a dependency ${\ul x}$ is {\em small} i.e. $|{\ul x}|\leq \om$ where $\om\to\infty$ slowly or (ii) ${\ul x}$ is {\em large} i.e. $|{\ul x}|=n/2+O(\sqrt{n\log n})$.
\end{theorem}

 A  gap property  in
solutions to random XOR-SAT systems over $GF(2)$ was previously observed by
Achiloptas and Molloy \cite{AMo}, and by Ibrahimi, Kanoria, Kraning and Montanari \cite{IKKM}. They found that
%, in the model where columns are variables (vertices) and rows are equations %(edges) and each column of the matrix in the system $Ax=b$ has $k$ random 1's.
the Hamming distance between XOR-SAT solutions was either $O(\log n)$ or at least $\a n$; where $n$ is the number of variables. In our case,  large dependencies  have intersection about $n/4$ (see Section \ref{back}), giving a precise value of $\a$.

Estimating the interaction between small and large dependencies is the main problem we solve. The negative correlation between the two types of  dependency is characterized by the binomial term in \eqref{Pmj}.

A dependency $\ul x$ is {\em fundamental} if there is no other dependency ${\ul y}\neq {\ul x}$ such that ${\ul y}\leq{\ul x}$, componentwise. We will prove in Section \ref{SmallLD} that the number $Z$ of fundamental small dependencies is asymptotically distributed as $Po(\f)$ i.e. Poisson with mean $\f$.
The quantity $P(\s,\l)$ in \eqref{result} is the limiting probability that small dependencies span a space of dimension $\s$, and large dependencies  span a space of dimension $\l$.

The distribution $\pi(k)$ given in \eqref{ldef} was previously observed in a model of random matrices over $GF(2)$ in which the entries $m_{i,j}$  are i.i.d Bernoulli random variables with $\Pr(m_{i,j}=1)=p$. For a wide range of $p$ the distribution of dimension $k$ of the null space  is given by $\pi(k)$.
The result was proved  by Kovalenko et al., \cite{KLS} for $p=1/2$, and extended
to the range $\min(p(n),1-p(n)) \ge (\log n +c(n))/n$, (where $c(n) \rai $ slowly)  by Cooper \cite{CCGF2}.
A similar result holds for the model of random matrices over the finite field $GF(t)$, see Cooper \cite{CCGFt}. Here the non-zero entries $\a \in GF(t)\sm\{0\}$ are independently and uniformly distributed with $\Pr(m_{i,j}=\a)=p/(t-1)$.  The distribution of co-rank $\pi_t(k)$ equivalent to  $\pi(k)=\pi_2(k)$ in \eqref{ldef} is obtained by directly replacing the $(1/2)$ terms in \eqref{ldef}  by $(1/t)$.

Finally we consider some related cases for $r$-out $s$-uniform hypergraphs.
For $r=1$ and $s=2$, $M$ has expected rank $\sim n- (\log n)/2$. This is because the expected number of components in a random mapping is $\sim (1/2) \log n$, (see e.g., \cite{FK}). Note: For $s$ even, the rows of $M$ add to zero modulo 2. The following theorem will be immediate from the proof of Theorem \ref{TH1}.
\begin{theorem}\label{th3}
If $r\ge 2$ and $s = 2,\;3$, then $M$ has rank $n^*=n-\ind{s=2}$, w.h.p.
\end{theorem}

Results for other finite fields follow easily from the analysis over $GF(2)$. We use the non-standard notation $GF(t)$ for a finite field of order $t$, rather than the usual $GF(q)$; and for brevity we consider only the \lq without replacement\rq ~case.  We consider three simple models with  entries from  the non-zero elements  of $GF(t)$, in each column. Let $\{f_i\}$ be a distribution on $i \in GF(t), i \ne 0$.
\begin{enumerate}[Model 1:]
\item The diagonal and other two non-zero entries in a column are 1.
\item The diagonal entries are 1, and the two other non-zero entries in each column are drawn u.a.r. from the  distribution $\{f_i\}$.
\item The diagonal and other two non-zero entries in each column are drawn u.a.r. from the uniform distribution $\{f_i\}$.
\end{enumerate}
For Model 2, let $\g=f_{t-1}$, $\a=\sum f_if_{t-i-1}$.
For Model 3, let $\g= \sum_i f_i f_{t-i}$, $\a=\sum_{i+j+k=0} f_if_jf_k$.

Let $\f_t$ be given by
\begin{equation}\label{fit}
\f_t=\sum_{\ell \ge 2} \frac 1\ell \brac{2\g e^{-2}}^\ell \sum_{i=0}^{\ell-2} \frac{\ell^i}{i!}.
\end{equation}
Because $M$ has 3 entries in each column, there is a special case of Model 1 for $GF(3)$.

 \begin{theorem}\label{TH2}
The following asymptotic results hold over $GF(t)$.
\begin{enumerate}
\item Model 1:
 If $t=3$ the limiting probability that $M$ has rank $n-1$ is 1, and $M$ has rank $n$ otherwise.
\item Models 2 and 3: If $t \ge3$ then provided $\a \le 2\g \le 1$,
\[
\Pr(\rank(M)=n-d) \sim \frac{\f_t^d}{d!} \;e^{-\f_t}.
\]
\end{enumerate}
\end{theorem}
In the simplest case  where entries are sampled uniformly from the non-zero elements of $GF(t)$, Theorem \ref{TH2}.2 holds for Models 2, 3 with $\g=1/(t-1)$.

{\bf Notation:}
Apart from $O(\cdot),o(\cdot),\Om(\cdot)$  as a  function of $n \rai$, we use the  notation $A_n \sim B_n$ if $\lim_{ n \rai} A_n/B_n=1$. The symbol $a \approx b$ indicates approximate numerical equality due e.g., to decimal truncation. The notation $\om(n)$ describes a function tending to infinity as $n \rai$. The expression {\em with high probability} (w.h.p.),   means with probability $1-o(1)$, where the $o(1)$ is a function of $n$, which tends to zero as $n \rai$.

\subsection*{Outline of the proof for $GF(2)$ with $r=1,s=3$}
Because the proofs are rather technical, {\em we give a detailed proof in the ``with replacement'' model}, and indicate separately in Section  \ref{ConV} why these results are also valid in the ``without replacement'' model. The difference in the range of summation indices for $\f_{\ol R}$ is explained in detail in Section \ref{SNRep}.

We refer to the rows of $M$ as $M_i,i\in [n]$ and to the columns as $C_j,j\in [n]$. By a set of rows $S$, we mean the set of rows $M_i,i\in S$.  A set of rows with  indices $L$ is linearly dependent (zero-sum) if $\sum_{i \in L} M_i=0 (\mod 2)$. A linear dependence  $L$  is {\em small} if  $|L| \le \om$, where $\om=\om(n)$ is a function tending slowly to infinity with $n$. A linear dependence $L$ is {\em large} if $|L| =(n/2)(1+O(\sqrt{\log n/n}))$. As part of our proof, we show that w.h.p. there are no other sizes of dependency. A set of zero-sum rows  $L$ is {\em fundamental} if $L$ contains no smaller zero-sum set and is disjoint from all other zero-sum sets. The zero-sum sets of size about $n/2$ are not disjoint. We count $k$-sequences of large dependencies with a property we call {\em simple}. Many of the problems with the proofs arise because the  large dependencies are not disjoint,  and are conditioned by the simultaneous presence of small linear dependencies in $M$.

We next outline the main steps in the proof of Theorem \ref{TH1}.
\begin{enumerate}
\item In Section \ref{SmallLD} we prove that the number $Z$ of small fundamental dependencies has factorial moments $\E (Z)_k \sim \f^k$, where $\phi$ is given by \eqref{phiNR}. Thus $Z$ is asymptotically Poisson distributed and
\[
\Pr\brac{\text{$M$ has $i$ small fundamental linear dependencies $\sim\frac{\f^i}{i!}e^{-\f}$}}.
\]
\item For $M \in \ul M(3,1,n)$ w.h.p. any fundamental sets of zero-sum rows of $M$ are either small (of size $\ell \le \om$) or large (of size $\ell =(n/2)(1+O(\sqrt{\log n/n}))$). This is proved in Section \ref{LargeFM}.
\item In Section \ref{HM} we discuss {\em simple} sequences of large dependencies, and in Section \ref{ESL} we estimate the moments of these sequences and determine their interaction with small dependencies.
\item We estimate the number of simple sequences, conditional on the the number of small fundamental dependencies. This leads to an approximate set of linear equations whose solution completes the proof of Theorem \ref{TH1}.
\end{enumerate}

\section{Small linear dependencies in $GF(2)$: with replacement}\label{SmallLD}

\paragraph{Notation}
 For $1\leq k \leq \om$, where $\om \rai$ arbitrarily slowly with $n$, let $X_k(M)$ or $Y_k(M)$ denote the number of index sets of $k$-dependencies in $M$. A $k$-dependency is {\em small} if $k \le \om$ and we use $Y_k$ when $k\leq \om$ and  use $X_k$ when $k \sim n/2$.
 We will show that for other values of $k$, $X_k=0$ w.h.p. We also use $Z_d,d\leq \om$ to denote the number $d$ of fundamental (minimal) dependent sets among the  rows of $M$.

We first consider dependencies with $s=o(n^{1/2})$ rows. For $S\subseteq [n]$, let $\cF(S)$ denote the event that the rows corresponding to $S$ are dependent. Let $Y_s$ denote the number of $s$-set dependencies.
\begin{lemma}
If $|S|=s=o(n^{1/2})$ then
\begin{equation}\label{EYL}
\Pr(\cF(S))\sim \bfrac{2s}{n}^s e^{-2s}.
\end{equation}
If $\omega\to\infty$, $\om\leq s=o(n^{1/2})$ then $Y_s=0$ w.h.p.
\end{lemma}
\begin{proof}
Suppose that $s=o(n^{1/2})$ and $S=[s]$. Then,
\begin{flalign}
\Pr(\cF(S))= & \brac{2 \bfrac{s}{n} \bfrac{n-s}{n}}^s \brac{\bfrac{s}{n}^2 + \bfrac{n-s}{n}^2}^{n-s} \nonumber\\
&\sim \bfrac{2s}{n}^s e^{-2s},\qquad \text{using }s=o(\sqrt n).\label{EXL}
\end{flalign}
{\bf Explanation:}
The probability that exactly one of the two random choices in a column of $S$ lies in a row of $S$ is $2 \bfrac{s}{n} \bfrac{n-s}{n}$.
The probability that both or neither of the two random choices in a column of $[n]\setminus S$ lies in a row of $S$ is $\bfrac{s}{n}^2 + \bfrac{n-s}{n}^2$.

This verifies \eqref{EYL}. It follows that
\[
\E(Y_s)\sim\binom{n}{s} \bfrac{2s}{n}^s e^{-2s}\sim \frac{(2s)^s e^{-2s}}{s!},
\]
As $\E Y_{s+1}/\E(Y_s)\sim 2/e$ we have that $\E Y_{\om}=e^{-\Omega(\om)}$ and so w.h.p. there are no dependencies with $\om\leq s =o(n^{1/2})$.
\end{proof}

Define $\s_s,\;\k_s$ by
\beq{kappal}{
\s_s=\sum_{j=0}^{s-1} \frac{s^j}{j!},
\qquad \text{ and } \qquad
\k_s= \frac{(s-1)!}{s^s}\s_s.
}
For $S\subseteq [n]$, let $\cF^*(S)$ denote the event that the rows corresponding to $S$ form a fundamental dependency. The next lemma deals with small fundamental dependencies.
\begin{lemma}
$\Pr(\cF^*(S)\mid\cF(S))=\k_s$.
\end{lemma}
\begin{proof}
The rows of the dependency $S$ consist of an $s \times s$ sub-matrix $M_{S,S}$ and a zero $(s \times n-s)$ sub-matrix. For $i \in S$, if $M_{i,i}=1$, then w.h.p. there is a unique entry $M_{j,i}=1$ which gives rise to an {\em edge}  $(i,j)$. If $M_{i,i}=0$ we regard this as a loop $(i,i)$. Thus $M_{S,S}$ is the incidence matrix of a random functional digraph $D_S$, and $S$ is fundamental iff  the underlying graph of $D_S$ is connected. For $s \ge 1$, $\Pr(D_S\text{ is connected})=\k_s$ (see e.g., \cite{Bo} or \cite{FK}).
\end{proof}
We now prove
\begin{lemma}\label{smalldisjoint}
Small fundamental dependent sets of $M$ are pairwise disjoint, w.h.p.
\end{lemma}
\begin{proof}
Let  $S,T$ be two small fundamental zero-sum row sets  with a non-trivial intersection $C=S \cap T$ and differences $A=S \sm T$, $\,B=T \sm S$, where $A \cup B\neq \emptyset$. Suppose $A \ne \es$. As the functional digraphs $D_S, D_T$ are connected, one of the following events must occur.  Either (i) some column of $C$ has two non-zero entries in the rows of $S \cup T$; or (ii) some column $j$ of  $A$ has %and/or  some column $j'$ of $B$ have
a non-zero entry in the rows of $C$. The latter is not possible as then a column of $S$
has a non-zero entry in the rows of $T$. Let $k=| S \cup T|$.
 The former has probability at most
%all of the above argument is needed. there are many special cases
\beq{zz1}{
%\sum_{a,b,c =1}^{\om}\binom{n}{a, b, c}\bfrac{O(\om)}{n}^{a+b+c+1}=o(1).
\sum_{k=2}^{2\om}\binom{n}{k}k\bfrac{k}{n}^{k-1}\bfrac{k}{n}^2=o(1).
}
\end{proof}
Given this lemma we can now prove
\begin{lemma}\label{smallsmall}
The number $Z$ of small fundamental dependent sets among the  rows of $M$ is asymptotically Poisson distributed with  parameter $\f_R$, and thus
\begin{equation}\label{Smallprob}
\Pr(Z=d) \sim \frac{\f_R^d}{d!} e^{-\f_R}.
\end{equation}
\end{lemma}
\begin{proof}
Fix $S\subseteq [n]$ and let $S_1,\ldots,S_d$ be a partition of $S$ with $|S_i|=s_i,i=1,2,\ldots,d$. Let $P(s_1,\ldots,s_d)$ be the probability that each $S_i,i=1,2,\ldots,d$ is a fundamental set, given that $S$ is a dependency. Thus,
\[
P(s_1,\ldots,s_d)= \frac{(s_1)^{s_1}\cdots(s_d)^{s_d}}{s^s}\prod_{i=1,...,d} \Pr(D_{S_i} \text{ connected})=\frac{1}{s^s} \prod_{i=1}^d (s_i-1)! \s_{s_i}.
\]
{\bf Explanation:} the factor $\frac{(s_1)^{s_1}\cdots(s_d)^{s_d}}{s^s}$ is the conditional probability that the random choices for columns with index in $S_i$ are in rows with index in $S_i$.

Thus, using \eqref{EYL}, we see that
\begin{flalign}\label{Zk1}
\E (Z)_d \sim &  \sum_{s \ge 1} \frac{ (2s)^s}{s!} e^{-2s} \sum_{s_1+\cdots+s_d=s}{ s \choose s_1,\ldots,s_d}P(s_1,\ldots,s_d)\\
= & \sum_{s \ge 1} \sum_{s_1+\ldots+s_d=s} \; \prod_{i=1}^d (2e^{-2})^{s_i} \frac{1}{s_i} \s_{s_i} \nonumber\\
=&\brac{\sum_{s \ge 1} \frac 1s (2e^{-2})^s \s_s}^d\nonumber\\
=&\f_R^d.\label{Zk2}
\end{flalign}
Thus, by the method of moments, the number of small disjoint fundamental zero-sum sets $Z$ tends tend to a Poisson distribution with parameter $\f_R$.
\end{proof}

\section{Large zero-sum sets: First moment calculations}\label{LargeFM}
Define an index set $J_a$ as follows,
\begin{equation}\label{what-the-L}
J_a=\{  n/2 - \sqrt{a n \log n} \le \ell \le n/2 + \sqrt{a n \log n}\}\text{ and }\overline{J}_a=[n]\setminus J_a, \,a\ge 0.
\end{equation}
\begin{lemma}\label{Th1}{(\bf Large linearly dependent sets.)}
Let $X_\ell$ denote the number of $\ell$-dependencies among the rows of $M$.
\begin{enumerate}[(i)]
\item $\sum_{\ell \in J_1} \E X_\ell \sim 1$.
\item Let $F=[n]\setminus  ([\om]\cup J_1)$, where $\om \rai$ arbitrarily slowly with $n$. Then  $\sum_{\ell \in F} \E X_\ell=o(1)$.
\end{enumerate}
\end{lemma}

\begin{proof}

%\vspace{-.5in}

From \eqref{EXL}, the expected number of dependencies of size $\ell$ is
\begin{flalign}\nonumber
\E X_\ell = & {n \choose \ell} \brac{2 \bfrac{\ell}{n} \bfrac{n-\ell}{n}}^\ell \brac{\bfrac{\ell}{n}^2 + \bfrac{n-\ell}{n}^2}^{n-\ell}.
\end{flalign}
We next approximate the expression for $\E X_\ell$. We note the following expansion.
\begin{equation}\label{logx}
(1+x)\log(1-x^2)+(1-x)\log(1+x^2)=-2 \brac{x^3+\frac{x^4}{2}+ \frac{x^7}{3}+\sum_{k\ge 4}\ind{k \text{ even}} \frac{x^{2k}}{k}\brac{ 1+\frac{k x^3}{k+1}}}.
\end{equation}
We write $\E X_\ell= {n \choose \ell} \F_\ell^n,\,\ell=(n/2)(1+\e)$, where
\begin{flalign}
\F_\ell =& \brac{\frac{1-\e^2}{2}}^{\frac{(1+\e)}{2}}
\brac{\bfrac{1+\e}{2}^2+\bfrac{1-\e}{2}^2}^{\frac{(1-\e)}{2}}\nonumber\\
=& \frac 12 (1-\e^2)^{\frac{(1+\e)}{2}}(1+\e^2)^{\frac{(1-\e)}{2}}\nonumber\\
=&\frac 12 \exp\left\{\frac 12\brac{(1+\e)\log(1-\e^2)+  (1-\e)\log (1+\e^2)}\right\}\nonumber\\
=&\frac 12 \exp\set{ -
\brac{\e^3+\frac{\e^4}{2}+ \frac{\e^7}{3}+\sum_{k\ge 4}\ind{k \text{ even}} {\e^{2k}}\brac{\frac 1{k}+\frac{\e^3}{k+1}}}} \nonumber\\
=&\frac 12 \exp\set{ -\brac{\e^3+\frac{\e^4}{2}+ \e_7}},
\label{good1}
\end{flalign}
where $|\e_7|\leq 2|\e|^7/3$ for sufficiently small $\e$.

Also for $\ell=(n/2)(1+\e)$,  $|\e|<1$,
\beq{Binomial}{
{n \choose \ell}= \brac{1+O\bfrac{1}{n}}\;\frac{2^n}{\sqrt{2\pi n(1-\e^2)}}\;\exp\brac{-n \brac{\frac{\e^2}{2}+ \frac{\e^4}{12}+\e_6}},
}
where $|\e_6|\leq |\e|^6/10$.

\paragraph{ Case 1: $\ell \in J_1$ .}

From \eqref{Binomial} with $|\e|=2\sqrt{(\log n)/n}$ we have
 \[
\frac{1}{2^n}\sum_{\ell \notin J_1} {n \choose \ell}=O(1/n^{5/2}),
\]
so that
\[
\frac{1}{2^n}\sum_{\ell \in J_1} {n \choose \ell}=1-O(1/n^{5/2}).
\]
Using \eqref{good1}, for $\ell \in J_1$,
 ${\F_\ell}^n=e^{\Th(n \e^3)}/2^n$. Then, as $n\e^3=O(\log^{3/2}n /\sqrt{n})$,
\[
\sum_{\ell \in J_1} \E X_\ell =\sum_{\ell \in J_1}{n \choose \ell} \frac{1}{2^n}
e^{\Th(n \e^3)}=1+o(1).
\]
For future reference, we note that for $|\e|<c<1$,
\begin{flalign}
\E X_\ell= & {n \choose \ell} \frac{1}{2^n}  \exp \set{-n\brac{\e^3+\frac{\e^4}{2}+\e_7}} \nonumber\\
=&\frac{\ooi}{\sqrt{2\pi n(1-\e^2)}}\exp\set{-n\brac{\frac{\e^2}{2}+\e^3+ \frac{\e^4}{2}+ \frac{\e^4}{12} +\e_6+\e_7}}\nonumber\\
=&\frac{\ooi}{\sqrt{2\pi n(1-\e^2)}}\exp\set{-\frac{n \e^2}{2}\brac{(1+\e)^2 +\frac{\e^2}{6}+O(\e^4)}}. \label{OK1}
\end{flalign}

%\newpage

\paragraph{Case 2: $\ell \in F$.}
Write $F=[n]\setminus  ([\om]\cup J_1)$ as  $F=F_1 \cup F_2\cup F_3$ where $F_1=\{\om,\ldots,3n/10\}$, $F_2=\{7n/10,\ldots,n\}$ and $ F_3=F\setminus (F_1\cup F_2)$. Thus, for $\ell \in F_3$,  $\ell=(n/2)(1+\e)$ where  $-2/5\leq\e\leq -\sqrt{(2 \log n)/n}$ or $\sqrt{(2 \log n)/n}\le \e \le 2/5$.

{\em Case $\ell \in F_1$.}
For sufficiently large $n$, Stirling's approximation implies that
\[
{n \choose \ell } \le \frac{n^n}{\ell^\ell(n-\ell)^{n-\ell}},
\]
so for some constant $C$ (in both with and without replacement models)
\beq{Xell}{
\E X_\ell \le\frac{ C n^n}{\ell^\ell(n-\ell)^{n-\ell}} \brac{2 \bfrac{\ell}{n} \bfrac{n-\ell}{n}}^\ell \brac{\bfrac{\ell}{n}^2 + \bfrac{n-\ell}{n}^2}^{n-\ell}.
 }
Continuing with this expression, using $\ell =\l n$ for $\l <1/2$,
\begin{flalign*}
\E X_\ell \le&  C \brac{ \frac{2^\l}{\l^\l (1-\l)^{1-\l}} \l^\l(1-\l)^{\l} (\l^2+(1-\l)^2)^{1-\l}}^n\\
=&  C \brac{2^\l(1-\l)^{\l}\brac{1-\l+\frac{\l^2}{1-\l}}^{1-\l}}^n\\
\le&  C \brac{2^\l(1-\l)^{\l} e^{-\l(1-\l)+\l^2}}^n\\
=& C \brac{2(1-\l)e^{-1+2\l}}^{\l n}\\
=& C [g(\l)]^{\l n}.
\end{flalign*}
The function $g(\l)$ is strictly concave and has a unique maximum at $\l=1/2$ with $g(1/2)=1$.
For $\l\le 3/10$, $g(\l)\le g(3/10)=(7/5)e^{-2/5}<1$ so that
\[
\sum_{\ell\in F_1}\E X_\ell \le C\sum_{\ell\in F_1}g(3/10)^{\ell} =o(1).
\]
{\em Case $\ell \in F_2$.}
Referring to \eqref{OK1},  the function $h(\e)=(\e^2/2)((1+\e)^2+\e^2 /6 +\e_6+\e_7)$ satisfies $h(\e)>2/25$ for $\e\geq 2/5$, and so
\[
\sum_{\ell \in F_2} \E X_\ell \le \sum_{\ell \in F_2} e^{-\Omega(n)}=o(1).
\]

{\em Case $\ell \in F_3$.}
For $\sqrt{ (2 \log n)/n} \le |\e| \le \sqrt{ (25 \log n)/n}$, the function $h(\e) \ge (1-o(1)) (\log n)/n$. Let $F_{3a}$ be the values of $\ell$ in this range
\[
\sum_{\ell  \in F_{3a}}\E X_\ell =O(\sqrt{n \log n})/n^{1-o(1)})=o(1/n^{1/3}).
\]
Let $F_{3b}=F_3 \sm F_{3a}$. Then $\e^2/2 \ge (25/2) (\log n)/n$, and $(1+\e)^2+\e^2 /6 +\e_6+\e_7 > 9/25$. Referring to \eqref{OK1},
\[
\sum_{\ell  \in F_{3b}}\E X_\ell =O(n)/n^4=o(1/n^{3}).
\]
\end{proof}

\section{Higher moments of large zero-sum sets: Background}\label{back}
Let $A \sd B$ denote the symmetric set difference of the sets $A$ and $B$. Thus
$A \sd B=(A \cup B)\sm (A \cap B)=(A \sm B) \cup (B \sm A)$.
 Suppose that, over $GF(2)$, the rows $M[i],i\in A$ indexed by $A$ are zero-sum,  thus $\ul z_A=\sum_{i \in A} M[i]=\ul 0$. Let $B$ be another  set such that $\ul z_{B}= \ul 0$. We can write $\ul z_A=\ul z_{A\ssm B}+\ul z_{A \cap B}$ and $\ul z_{B}=\ul z_{B\ssm A}+\ul z_{A \cap B}$. Adding these two sets of rows modulo 2 has the effect of canceling the intersection $A \cap B$. Thus (i) $\ul z_A+\ul z_{B}=0$, whether $\ul z_{A \cap B}$ is itself zero-sum or not; and (ii) $\ul z_A+\ul z_{B}=\ul z_{A \oplus B}$.

Recall that a set of zero-sum rows is fundamental if it contains no smaller zero-sum set of rows. For small sets we were able to count fundamental dependencies directly. We have to adopt an alternative strategy for  large zero-sum sets. We use an approach similar to the one given in \cite{CCGF2}. We count {\it simple} sequences of large linearly dependent row sets $B=(B_1,...,B_k)$,  $k\ge 1$ constant, and where $|B_i| \in J_1$ so that $|B_i|\sim n/2$. A $k$-tuple of large dependent sets $ B=(B_1,...,B_k)$ is simple, if for all sequences $(j_1<j_2<...<j_l)$ and $ (1\le l \le k)$ the set differences satisfy
\beq{xx1}{
 |B_{j_1}\sd  B_{j_2}\sd\cdots \sd  B_{j_l}|  \in J_1
}
For any given matrix $M$ there is a largest $k$ such that $B_1,...,B_k$ are simple. In which case, we say $k$ is {\it maximal} and  $B_1,...,B_k$ is a {\it maximal simple sequence}.

Let $V(M)=\{\emptyset\}\cup\{B:B \mbox{ is zero-sum in } M\}$, then  $(V(M),\oplus)$ is a vector space over $ {GF}_2$ under the convention that $0 \cdot B=\emptyset,\; 1\cdot B=B$. In $V(M)$ a simple sequence $(B_1,...,B_k)$ is an ordered basis  for a subspace $S$ of dimension $k$.

Given $k$ linearly dependent sets of rows with index sets $B_1,\cdots,B_k$, there are $2^k$ intersections of these sets and their complements. For each $\ux=(x_1,\cdots,x_k)$, $\ux \in \{0,1\}^k$ we let $I_\ux=\cap_{i=1,...,k}B_i^{(x_i)}$ where $B_i^{(0)}=\ol B_i=[n] \setminus B_i$ and $B_i^{(1)}=B_i$. The index sets $I_\ux$ are disjoint by definition and their union (including $\ux_0=(0,\cdots,0)$) is $[n]$.

Next for $\ux \in \{0,1\}^k$ let $B(\ux)={\bigoplus}_{i:x_i=1}B_i$. Let $K=2^k-1$. Let $U$ be a $K\times K$ matrix  indexed by $\ux,\uy \in \{0,1\}^k$, $\ux,\uy \ne 0$; with entries $U(\ux,\uy)=1$ if $I_\uy\subseteq B(\ux)$, and $U(\ux,\uy)=0$ otherwise. In summary,
\begin{align*} %\label{Mx}
&\text{Row index } \ux=(x_1,x_2,\ldots,x_k)\text{ is the indicator vector for } \displaystyle{B(\ux)={\bigoplus}_{i:x_i=1}}B_i, %\label{My}
\\
&\text{Column index }\uy=(y_1,y_2,\ldots,y_k) \text{ is the indicator vector for   }I_\uy=\bigcap_{i=1,...,k}B_i^{(y_i)}.
\end{align*}
The row of $U$ representing the set $B(\ux)$  is formed by adding the rows of those sets $B_i$ such that $x_i=1$ in $\ux$; the addition being over $GF(2)$.
Thus $B(\ux)$ is the union of the sets $I_\uy$, where $y_i=1$ for an odd number of those  sets $B_i$ where $x_i=1$. This can be seen inductively by generating $B_1$, $B_1 \sd B_2$, $(B_1 \sd B_2)\sd B_3$ etc. in the given order. In summary $U(\ux,\uy)=1$ iff both $x_i=1$ and $y_i=1$  for an odd number of indices $i$, and thus, over $GF(2)$,
\beq{xiyi}{
U(\ux,\uy)=\sum_{i=1}^k x_i y_i.
}
Our aim  is to use $U$, treated as a real matrix to show that w.h.p. $|I_\ux|\sim n/2^k$ for every $\ux$.
We do this by observing that given the characterisation $U(\ux,\uy)=1_{I_{\uy}\subseteq B(\ux)}$, the vector $(|I_\ux|,\,\ux\in\set{0,1}^k,\,\ux\neq 0)$ is the solution ${\ul z}$ over the reals of an equation
\beq{Uz}{
U{\ul z}={\ul b}\text{ where }{\ul b}\sim \frac{n}2{\bf 1},
}
assuming that $ B=(B_1,...,B_k)$ is simple. To prove that $|I_\ux|\sim n/2^k$, we  prove the properties of $U$ listed in Lemma \ref{L77} below.

Equation \eqref{xiyi} implies that by arranging the rows and column indices of $U$ in the same order, $U$ will be symmetric. We will choose an ordering such the first $k$ rows  correspond to $B_i, i=1,...,k$. Thus $x_i=e_i,i=1,2,\ldots,k$ where $e_1=(1,0,\ldots,0)$ etc., and $y_i=e_i,i=1,2,\ldots,k$.
After this we let $Q$ be the $k \times K$ matrix with column indices $x$ made up of the first $k$ rows. Thus row $i$ represents $B_i, i=1,...,k$ and $U$ contains a $k\times k$ identity matrix in the first $k$ rows and columns.

The row indexed by $\ux=(x_1,...,x_k)$ is the linear combination $\sum_{i=1}^k x_i \ul r_i$ of the rows of $Q$, and  corresponds to $B(\ux)$ in the vector space $V(M)$ given above.

\begin{lemma}\label{L77}
The $K \times K$ matrix $U$ has the following properties:
\begin{enumerate}[(i)]
\item The matrix $U$ symmetric.
\item Every row or column of $ U$ has $2^{k-1}$ non-zero entries.
\item Any two distinct rows of  $ U$ have $2^{k-2}$ common non-zero entries.
\item The matrix $U$ is non-singular when the entries are taken to be over the real numbers, and the matrix $S=UU^{\top}=U^2=2^{k-2}(I+J)$ is symmetric, with inverse $S^{-1}=(1/2^{k-2})(I-J/2^k)$; where $J$ is the all-ones matrix.
\end{enumerate}
\end{lemma}
\begin{proof}
(i) This follows immediately from \eqref{xiyi}, and the above construction.

(ii) Fix $\ux$ and assume that $x_1=1$. There are $2^{k-1}$ choices for the values of $y_i,i=2,3,\ldots,k$. Having made such a choice, there are two choices for $y_1$, exactly one of which will give $\sum_{i=1}^kx_iy_i=1$.

(iii) Fix $\ux,\ux'$ and think of rows $\ux,\ux',\ux+\ux'$ as non-empty subsets of $[2^k]$. Then we have $|\ux|=|\ux'|=|\ux\setminus \ux'|+|\ux'\setminus \ux|=2^{k-1}$, by (iii). Thus $|\ux|+|\ux'|-(|\ux\setminus \ux|+|\ux'\setminus \ux|)=2|\ux\cap \ux'|=2^{k-1}$.

(iv) That the matrix $ U$ is non-singular over the real numbers, uses an argument given in \cite{BR} (pages 11-13). Let $S=UU^{\top}$. Let $\ul u, \ul v$ be distinct rows of $ U$, then $\ul u \cdot \ul u=2^{k-1}$ and $\ul u \cdot \ul v=2^{k-2}$. Thus $S=2^{k-2} (I+J)$, where $J$ is the all-ones matrix. The reader can check that $S^{-1}=\frac{1}{2^{k-2}}(I-\frac{1}{2^{k}}J)\;2^{k-1}$ which implies that $U$ is invertible too.
\end{proof}
The definition of a simple $k$-tuple $(B_1,...,B_k)$ requires that all sets $B_i$ be large and their set differences to be distinct and of size $\sim n/2$. Let $(|B_1|,\ldots,|B_k|) \sim (n/2)\ul 1$ be the vector of these set sizes. Over the reals, solving \eqref{Uz} gives the sizes of the subsets $I_\ux$.

\begin{lemma}\label{Cor4}
Let $(B_1,...,B_k)$ be a simple sequence.   Then for all $\ux \in \{0,1\}^k$,
\beq{Ixin}{
|I_\ux|= \frac{n}{2^k}\brac{1\pm { 4^k}\sqrt{\frac{\log n}{n}}}.
}
\end{lemma}
\begin{proof}
{  Let $i=1,...,K$ index the rows of $U$, and $j =1,...,K$ index the  columns. Let $B(i)$ be the set corresponding to the row $i$ of $U$. Referring to \eqref{Uz}, let $\uy=(2/n)\uz$, and  $U \uy=\ul b$ where  now $b_i=2|B(i)|/n=1+\e_i$, so that $|\e_i|\leq 2\sqrt{\log n/n}$. The matrix  $S=U^2$, so  $S \uy=U \ul b=\ul c$ where $c_i=2^{k-1}(1+\d_i)$ and $\d_i= \sum_{j:U(i,j)=1} \e_j/2^{k-1}$, the summation being over the $2^{k-1}$-subset of non-zero entries of row $i$ of $U$.} Thus, as $J$ is $K \times K$ where $K=2^{k}-1$,
\[
\uy=S^{-1} \ul c= \frac{1}{2^{k-2}}\brac{I-\frac{1}{2^{k}}J}\;2^{k-1}(\ul 1 +\ul \d)= \frac{1}{2^{k-1}}\ul 1 +\ul \eta,
\]
where $|\ul \eta|\leq 2^{k}\sqrt{\log n/n}$. It follows that w.h.p. the solution { $\uz$ to \eqref{Uz} over the real numbers} satisfies $|I_\ux|=(n/2^k)(1\pm { 4^k} \sqrt{\log n/n})$ for all $\ux\in \set{0,1}^k$.
\end{proof}

\begin{remark}\label{remM}
The proofs above
generalize to the case where $\ul b\sim (\xi n,\xi n,\ldots,\xi n)$ for some constant $\xi \in (0,1/2]$ in equation \eqref{Uz}. In which case
\eqref {Ixin} becomes
\[
|I_\ux|=  \frac{2\xi n}{2^k}\brac{1\pm 2^k\sqrt{\frac{\log n}{n}}}.
\]
\end{remark}

\section{Simple sequences of large zero-sum sets.}\label{HM}

Let $B_1,B_2,\dots,B_k$ be a simple sequence. In row $M_i$ of the matrix $M$, there is a 1 in the diagonal entry $M_{i,i}$. As $s=3$  there need to be two (random) 1's in column $C_i$ chosen in a way to ensure the linear dependence of $B_1,\ldots,B_k$. The following lemma describes where these non-zeros must be placed.
\begin{lemma}\label{condist}
$B_1,\cdots,B_k$ are dependencies if and only if the following holds for all $i\in [n]$. Suppose that row $i$ is in $ I_\ux$, and that the two random non-zeros $e_1(i),e_2(i)$ in column $i$ are in $I_\uu,I_\uv$ respectively. Then we must have $\ux=\uu+\uv(\mod 2)$.
\end{lemma}

\begin{proof}
Let $\ux=(x_1,...,x_k)$ and consider $x_m$ for $1\leq m\leq k$. If $x_m=0$ then $i\notin B_m$, so either none or both of $j,j'$ are in $B_m$, and so zero or two unit entries in this column are in $B_m$. We must therefore have either $u_m=v_m=0$ or $u_m=v_m=1$ and $x_m=u_m+v_m$. If $x_m=1$ then $i \in B_m$ and so exactly one of $e_1(i),e_2(i)$ must also be in $B_m$. Hence $u_m=1,v_m=0$, or vice versa. Thus in all cases $x_m=u_m+v_m$.
\end{proof}

The main result of this section is the following.
\begin{lemma}\label{EXk}
Let $k \ge 1$ be a positive integer, and let $\bX_k$ count the number of simple $k$-sequences of large dependencies. Then $\E(\bX_k)\sim  1.$
\end{lemma}

\begin{proof}
We have to estimate the expected number of simple sequences $(B_1,...,B_k)$ of large dependencies. By \eqref{Ixin} of Lemma \ref{Cor4} the index sets $I_\ux$  have size  $|I_\ux|=(n/2^k)(1+ O(\sqrt{\log n/n}))$. Let $K=2^k-1$ as above, and let
\[
\Om=\set{\ul h =( h_0,h_1,...,h_K) : h_i \text{ satisfies \eqref{Ixin}}, \sum_{i=1}^K h_i \in J_1}.
\]
Then we define $\Phi(\ul h, k)$ by
\begin{flalign}
\E(\bX_k)= &\sum_{\ul h \in \Om} {n \choose h_0, h_1, \ldots, h_K} \prod_{\ux \ne 0}\brac{2 \sum_{\{\uu,\uv\} \atop \uu+\uv=\ux}\frac {h_\uu}{n} \frac{h_\uv}{n}}^{h_\ux}\brac{\sum_\uu \bfrac{h_\uu}{n}^2}^{h_0}\label{mul1}\\
=&\sum_{\ul h \in \Om} {n \choose h_0, h_1, \ldots, h_K} \Phi(\ul h, k).
\label{multinomial}
\end{flalign}
{\bf Explanation of \eqref{mul1}}.
Let $h_\ux=|I_\ux|$. The multinomial coefficient ${n \choose h_0, h_1, \ldots, h_K}$ counts the number of choices for the subsets $I_\ux$. In the product, in order for $B_1,...,B_k$ to be zero-sum, for $\ux \ne 0$  we need to cancel the diagonal entries $M_{j,j}=1$ of $j \in I_x$ within the columns indexed by $I_\ux$. This is achieved by putting one entry in rows $I_\uu$ and one in rows $I_\uv$ where  $\uu+\uv=\ux$. The last factor counts the choices for the entries of columns indexed by $I_0$ over the row index sets $I_\uu$, either zero or two in an index set, in order to preserve the zero-sum property.

Set $h_\ux=(n/2^k)(1+\e_\ux)$ where $|\e_\ux|=O(\sqrt{\log n/n})$. We note that  $\sum_\ux \e_\ux=0$, implies that
\[
\sum_{\ux}h_\ux\e_\ux=\frac{n}{2^k}\sum_\ux(\e_\ux+\e_\ux^2)=\frac{n}{2^k}\sum_\ux\e_\ux^2\text{ and } \sum_{\ux}h_\ux\e_\ux^2=\frac{n}{2^k}\sum_\ux\e_\ux^2+O\bfrac{\log^{3/2}n}{n^{1/2}}.
\]
And then Stirling's approximation implies that
\begin{align*}
{n \choose h_0, h_1, \ldots, h_K}&\sim \frac{n^n\sqrt{2\p n}}{\prod_{\ux\in \set{0,1}^k}((n/2^k)(1+\e_\ux))^{h_\ux} (\sqrt{2\p n/2^k})^{2^k}}\\
&=2^{kn}\exp \set{-\sum_{\ux\in \set{0,1}^k}^Kh_\ux\brac{\e_\ux-\frac{\e_\ux^2}{2}}+O(\log n)}\\
&=2^{kn}\exp \set{-\frac{n}{2^{k+1}}\sum_{\ux\in \set{0,1}^k}^K\e_\ux^2+O(\log n)}=2^{kn}n^{O(1)}.
\end{align*}
In addition, by considering random $2^k$-colorings of $[n]$ we see from the Chernoff bounds that
\begin{equation}\label{multinomial1}
\sum_{\ul h \in \Om}{n \choose h_0, h_1, \ldots, h_K}= 2^{kn}(1-O(n^{-2^k/3})).
\end{equation}
With respect to \eqref{mul1}, using $\sum_\ux \e_\ux=0$, we see that
\begin{flalign}
\brac{\sum_{\uu\in \set{0,1}^k} \bfrac{h_\uu}{n}^2}^{h_0} =&\brac{\sum_\uu \frac{1}{2^{2k}} (1+2\e_\uu+\e_\uu^2)}^{h_0} \nonumber\\
=&\bfrac{1}{2^{k}}^{h_0} \brac{1+\frac{1}{2^k}\sum_\uu \e_\uu^2}^{h_0}\nonumber\\
=&\bfrac{1}{2^{k}}^{h_0} \exp\set{\frac{n}{2^k}(1+\e_0) \log \brac{1+\sum_\uu \frac{\e_\uu^2}{2^k}}}\nonumber\\
&=\bfrac{1}{2^{k}}^{h_0} \exp\set{\frac{n}{2^{2k}} \sum_\uu \e_\uu^2+O\bfrac{\log^{3/2}n}{n^{1/2}}}.\label{epsu}
\end{flalign}
If $\ux \ne 0$ then each index ${\ul z}$ occurs exactly once in $\sum_{\{\uu,\uv\} \atop \uu+\uv=\ux}(\e_\uu+\e_\uv)$ and so $\sum_{\{\uu,\uv\} \atop \uu+\uv=\ux}(\e_\uu+\e_\uv)=\sum_{\ul z} \e_{\ul z}=0$. Therefore,
\begin{flalign*}
\brac{2 \sum_{\{\uu,\uv\} \atop \uu+\uv=\ux}\frac {h_\uu}{n} \frac{h_\uv}{n}}^{h_\ux}
=&\brac{2 \sum_{\{\uu,\uv\} \atop \uu+\uv=\ux} \frac{1}{2^{2k}}(1+\e_\uu+\e_\uv+\e_\uu\e_\uv)}^{h_\ux}\nonumber\\
=&\bfrac{1}{2^k}^{h_\ux} \brac{1+\frac{1}{2^k} \sum_{\{\uu,\uv\} \atop \uu+\uv=\ux} 2\e_\uu\e_\uv}^{h_\ux} \\
=&\bfrac{1}{2^k}^{h_\ux}\exp\set{\frac{n}{2^k}(1+\e_\ux) \log \brac{1+2\sum_{\{\uu,\uv\} \atop \uu+\uv=\ux} \frac{\e_\uu\e_\uv}{2^k}}}\\
&=\bfrac{1}{2^k}^{h_\ux}\exp\set{\frac{n}{2^{k}}\sum_{\{\uu,\uv\} \atop \uu+\uv=\ux} \frac{2\e_\uu\e_\uv}{2^k}+O\bfrac{\log^{3/2}n}{n^{1/2}}}.
\end{flalign*}
 Note that
\[
\Lambda =\sum_{\ux \ne 0} \sum_{\{\uu,\uv\} \atop \uu+\uv=\ux} 2\e_\uu\e_\uv=
\sum_\uu \e_\uu \sum_{\ux+\uu \atop \ux \ne 0}\e_{\ux+\uu}=\sum_\uu\e_\uu\sum_{\uv \ne \uu}\e_\uv,
\]
gives
\[
\Lambda +\sum_\uu \e_\uu^2 = \brac{\sum_\uu \e_\uu}^2=0.
\]
Thus using $\sum_\ux h_\ux=n$,
\begin{flalign}
\Phi(\ul h, k)=& \bfrac{1}{2^k}^{\sum_\ux h_\ux}
\exp\set{ \frac{n}{2^{2k}}\brac{\sum_\uu \e_\uu^2 +\sum_{\ux \ne 0} \sum_{\{\uu,\uv\} \atop \uu+\uv=\ux} 2\e_\uu\e_\uv }+ O\bfrac{\log^{3/2}n}{n^{1/2}}}\nonumber\\
=&\frac{1}{2^{kn}} e^{ O({\log^{3/2}n/\sqrt{n}})}.\label{this1}
\end{flalign}
It follows from \eqref{multinomial}, \eqref{multinomial1} and \eqref{this1} above that
\begin{equation}\label{error}
\E( \bX_k)=1+O\bfrac{\log^{3/2}n}{\sqrt n}=1+o(1).
\end{equation}
\end{proof}

\section{Conditional expected number of small zero-sum sets}\label{ESL}
Let  $(B_1,\ldots,B_k)$ be a fixed sequence of subsets of $[n]$ with $|B_i|\in J_1$ for $i=1,2,\ldots,k\leq \om$. Let $\cB$ be the event
\begin{equation}\label{Bee}
\cB=\set{(B_1,...B_k)\text{ is a simple sequence of large row dependencies}}.
\end{equation}

\begin{lemma}\label{condista}
Given $\cB$ and $i\in I_\ux$, { $|I_\ux|=h_\ux$,} the distribution of the row indices $\ell,\ell'$ of the other two non-zeros in column $i$ is as follows.
 \\
 If $\ux\neq 0$ then choose $\uu,\uv$ such that $\ux=\uu+\uv\ \mod 2$ with probability
\[
p(\uu,\uv)={ \frac{h_\uu h_\uv}{\sum_{\uy+\uz=\ux}h_\uy h_\uz},}
\]
and then randomly choose $\ell\in I_\uu,\ell'\in I_{\uv}$. If $\ux=0$ then choose $\uu$ with probability
\[
p(\uu,\uu)=\frac{h_\uu^2}{\sum_{\uy\in \set{0,1}^k}h_\uy^2},
\]
and then randomly choose $\ell,\ell'\in I_\uu$.
\end{lemma}
\begin{proof}
This follows from the fact that the non-zeros in each column are independently chosen with replacement and from the condition given in Lemma \ref{condist}.
\end{proof}
{
For $m \le \om$, let $S_j,\;j=1,2,\ldots,m$ be  pairwise disjoint  subsets of the rows of $M$, where $|S_j| \le \om$. Let $S=\bigcup_{j=1}^m S_j$ and $s=|S|$. For $j=1,2,\ldots,m$  define the following events
\begin{align*}
\cS_j=\{ S_j\text{ is a small zero-sum set}\}, \quad
\cS_j^*=\{ S_j\text{ is a small fundamental zero-sum  set}\}.
\end{align*}
Let
\[
\cS=\bigcap_{j=1}^m\cS_j \qquad\text{ and }\qquad\cS^*=\bigcap_{j=1}^m\cS_j^*.
\]
}

We need to understand the conditioning imposed by {  the event $\cB$ in \eqref{Bee} on the small dependencies.}

\begin{lemma}\label{smallcond}
\begin{equation}\label{PFL}
\Pr( \cS^*\mid  \cB)\sim \Pr(\cS^*).
\end{equation}
\end{lemma}

\begin{proof}
{ Let $I_\ux$, $\ux \in \{0,1\}^k$, be as defined in Section \ref{back}.
Let $h_\ux=|I_\ux|$. By Lemma \ref{Cor4} we can assume that $|I_\ux|=h_\ux \sim n/2^k$ for all $\ux\in \set{0,1}^k$.
For  $j=1,2,\ldots,m$, let $S_{j,\ux}=S_j \cap I_\ux$ and $s_{j,\ux}=|S_{j,\ux}|$.
Similarly, let $S_\ux=S\cap I_\ux$,  $s_\ux=|S_\ux|$.  These definitions include $\ux=\bo$, so that
 $S_{\bo}=I_\bo\cap S$ and $s_{j,\bo}=|S_{j,\bo}|$ etc.}

For each $i \in [n]$, we consider the probability that column $i$ of $M$ is consistent with $\cS$ according to four cases.
\paragraph{{ Case 1: $i\in I_{\bo}\setminus S$.}}
For each column $i\in { I_\bo \sm S= I_{\bo}\setminus S_{\bo}}$, we must estimate the probability that the two non-zeros $e_1(i),e_2(i)$ are in rows consistent with the occurrence of $\cS$.  Because $i \in I_{\bo}$ and $\cB$ occurs, we know from Lemma \ref{condist} that $e_1(i),e_2(i)\in I_\uu$ for some $\uu \in \{0,1\}^k$. For $\cS$ to occur, we require that zero or two of  $e_1(i),e_2(i)$ fall in $S_\uu$, an event of conditional probability $(1-s_\uu/h_\uu)^2+(s_\uu/h_\uu)^2$.

Let $E_\uu$ denote the number of non-zero pairs from $I_{\bo}\setminus S_{\bo}$ falling in $I_\uu$. Then the conditional probability that the non-zeros of $I_{\bo}\setminus S_{\bo}$ are consistent with $\cS$ is given by
\beq{L00}{
\Pr(I_{\bo}\setminus S_{\bo} \text{ is consistent } \cS\mid\cB) = \E\brac{\prod_\uu \brac{ 1-2\frac{s_\uu}{h_\uu}+2\bfrac{s_\uu}{h_\uu}^2}^{ E_\uu}}.
}

Given $\cB$, we see that $E_\uu$ is  distributed as $Bin(h_{\bo}-s_{\bo}, p(\uu,\uu))$, and has expectation
\[
\E(E_\uu)= (h_{\bo}-s_{\bo}) \frac{ h_\uu^2}{h_{\bo}^2+h_1^2+\cdots +(h_{2^k-1})^2} \sim \frac{h_{\bo}}{2^k}.
\]
{ By Lemma \ref{Cor4} we can assume that $h_\bo \sim N=n/2^k$.}
The Chernoff bounds imply that $E_\uu$ is concentrated around its mean $(h_{\bo}-s_{\bo})p(\uu,\uu)$.  Thus,
\beq{suff}{
\left|{E_\uu-\frac{h_{\bo}}{2^k}}\right|\leq n^{2/3}\quad\text{with probability at least }1-e^{-\Omega(n^{1/3})}.
}
Going back to \eqref{L00} and using \eqref{suff} gives
\mult{L0}{
\Pr(I_{\bo}\setminus S_{\bo} \text{ is consistent with the occurrence of } \cS\mid\cB) \sim\\ \prod_\uu\brac{1-\frac{2s_\uu}{N}}^{N/2^k} \sim \exp\set{-2\sum_\uu \frac{s_\uu}{2^k}}=e^{-s/2^{k-1}}.
}

\paragraph{{ Case 2: $i\in I_\ux\setminus S$, $\ux \ne \bo$.}}
Given $\cB$, and $i \in I_\ux$, { we know from Lemma \ref{condist}} that the non-zeros $e_1(i),e_2(i)$ of column $i$ lie in $I_\uu, I_{\ux+\uu}$ respectively, { for some} $\uu \in \{0,1\}^k$. The probability of this is $p(\uu,\ux+\uu)$. The number $E_\ux(\uu,\ux+\uu)$ of such pairs of non-zeros in $I_\uu, I_{\ux+\uu}$ has distribution $Bin((h_\ux-s_\ux) p(\uu,\ux+\uu))$, and expectation asymptotic to  $(h_\ux-s_\ux)/2^{k-1}$.

The rows of $S_1,\ldots,S_m$ have to be zero-sum in this column, so either exactly one non-zero falls in { some}  $S_{j,\uu},S_{j,\ux+\uu}$ for some $1\leq j\leq m$ or  exactly one non-zero falls in { some} $I_\uu\setminus S_\uu,I_{\ux+\uu}\setminus S_{\ux+\uu}$. The { conditional} probability of this  is
\begin{flalign*}
P(\uu,\ux+\uu)=&\E\brac{\brac{\sum_{j=1}^m
\frac{s_{j,\uu}}{h_\uu}\frac{s_{j,\ux+\uu}}{h_{\ux+\uu}} +\frac{h_\uu-s_\uu}{h_\uu} \frac{h_{\ux+\uu}-s_{\ux+\uu}}{h_{\ux+\uu}}}^{E_\ux(\uu,\ux+\uu)}}
\\
\sim&\brac{\sum_{j=1}^m \frac{s_{j,\uu}s_{j,\ux+\uu}}{N^2} +\frac{N-s_\uu}{N}\frac{N-s_{\ux+\uu}}{N}}^{(N-s_\ux)/2^{k-1}}\\
& \sim e^{-(s_\uu+s_{\ux+\uu})/2^{k-1}}.
\end{flalign*}
For a given $\ux$ there are $2^{k-1}$ unordered pairs $S_\uu, S_{\ux+\uu}$, so
\beq{sumell}{
\Pr(I_\ux\setminus S_\ux \text{ is consistent with } \cS) \sim \exp\set{-\frac{1}{2^{k-1}}\sum_{\{u,\ux+\uu\}} (s_\uu+s_{\ux+\uu})}= e^{-s/2^{k-1}}\;.
}
Note that, in the sum in \eqref{sumell} $s_\uu+s_{\ux+\uu}$ and $s_{\ux+\uu}+s_{\uu}$, contribute as one term.
Thus
\begin{equation}\label{ALX}
\Pr( I_\ux\setminus S_\ux \text{ is consistent with } \cS,\forall \ux \ne \bo)\sim e^{-(2^k-1) s/2^{k-1}}.
\end{equation}

\paragraph{Case 3: $i\in S_{j,\ux} \subseteq I_\ux$, $\ux \ne 0$.}
Suppose that the pair $e_1(i), \, e_2(i)$ fall in $I_\uu, I_{\uu+\ux}$.
For $i \in S_{j,\ux}$, one  non-zero needs to be in $S_j$, and the other to { completely avoid $S$.}
Let $\uv=\ux+\uu$. The probability  this happens  is
\beq{Pjx}{
P_j(\uu,\uv)\sim \frac{1}{2^{k-1}}\brac{ \frac{s_{j,\uu}}{h_\uu}
\frac{h_\uv-{ s_\uv}}{h_\uv}+
\frac{s_{j,\uv}}{h_\uv}\frac{h_\uu-{ s_\uu}}{h_\uu}}.
}
The events $\set{\uu,\ux+\uu}$ are disjoint and { are an exhaustive dissection of $S_j$}. For a given $i \in S_{j,\ux}$, the probability $p(i,j)$ of success   is
\mult{dab}{
p(i,j) =\sum_{\{\uu,\uu+\ux\}} P_j(\uu,\uu+\ux)\sim \frac{1}{2^{k-1}} \sum_{\uu,\uv=\ux+\uu}\brac{ \frac{s_{j,\uu}}{N} \frac{N-{ s_{\uv}}}{N}+ \frac{s_{j,\uv}}{N}\frac{N-{ s_{\uu}}}{N}}\\
\sim \frac{s_j }{N 2^{k-1}}\brac{1+O\bfrac{\om}{N}}.
}
Every column of $S_{j,\ux}$ has to succeed or { some $S_t$} is not a small zero-sum set. Thus
\[
\Pr({ S_{j,\ux}} \text{ succeeds})\sim \bfrac{s_j (1+O(s/N))}{N 2^{k-1}}^{s_{j,\ux}}.
\]
As { $\sum_{\ux \ne \bo} s_{j,\ux}=s_j-s_{j,\bo}$,} the above allows us to calculate
\begin{equation}\label{LL}
{ \Pr(S_{j,\ux} \text{ succeeds }\forall\ux \ne \bo)} \sim \bfrac{s_j}{N 2^{k-1}}^{s_j-s_{j,\bo}}.
\end{equation}
\paragraph{Case 4: $i\in S_{j,\bo} \subseteq { I_{\bo}}$.}
In the case that $\ux=\bo$, and $S_{j,\bo} \subseteq I_{\bo}$, the non-zeros in a column of $S_{j,\bo}$ must both fall in the same index set $I_\uu$; one in $S_{j,\uu}$ and one in $I_\uu\setminus S_{j,\uu}$.  Thus $P(\uu,\uu)$ is now summed over all $I_\uu$, a total of $2^k$ such sets. For $i \in S_{j,\bo}$, the probability $p(i)$ of success is
\[
p(i) =\sum_{\{\uu,\uu\}} P(\uu,\uu)\sim \frac{1}{2^{k}} \sum_{\uu}\brac{2 \frac{s_{j,\uu}}{N} \frac{N-s_{j,\uu}}{N}}
\sim \frac{s_j}{N 2^{k-1}} \brac{1+O\bfrac{\om}{N}}.
\]
The final term is the same as in \eqref{dab}, and we obtain
\beq{36}{
{ \Pr(S_{j,\bo} \text{ succeeds})\sim \bfrac{s_j}{N 2^{k-1}}^{s_{j,\bo}}}
}
Using \eqref{L0}, \eqref{ALX}, \eqref{LL} and \eqref{36}, we obtain
\beq{crispy}{
\Pr(\cS\mid\cB) \sim \prod_{j=1}^m\bfrac{s_j}{N 2^{k-1}}^{s_j} e^{- (2^k-1) s/2^{k-1}} e^{-s/2^{k-1}} =\prod_{j=1}^m\bfrac{2s_j}{n}^{s_j} e^{-2s}.
}
{ Applying \eqref{EXL} to the right hand side of \eqref{crispy}}
completes the proof of $\Pr(\cS\mid\cB)\sim\Pr(\cS)$. To replace $\cS$ by $\cS^*$  the conditional probability that $S_j$ is fundamental is obtained by multiplying by $\k_{s_j}$ of \eqref{kappal}. This completes the proof of the lemma.
\end{proof}

{ We can now use inclusion-exclusion to prove the following lemma.}
\begin{lemma}\label{dimsmall}
Let $\S_\s$ be the event that there are exactly $\s$ disjoint small fundamental dependencies. Then,
\[
\Pr(\S_\s\mid\cB)\sim\frac{\f_R^\s e^{-\f_R}}{\s!}\sim \Pr(\S_\s).
\]
\end{lemma}
\begin{proof}
{ Let $s=s_1+\cdots+s_\ell$, then}
\begin{align*}
T_\ell=&\frac{1}{\ell!}\sum_{1\leq s_1,\ldots,s_\ell\leq\om}\sum_{|S_i|=s_i, \atop i=1,\ldots,\ell} \Pr\brac{\bigcap_{i=1}^\ell \cS_i^*\bigg|\cB}\sim\frac{1}{\ell!}\sum_{1\leq s_1,\ldots  s_\ell\leq\om}\sum_{|S_i|=s_i, \atop i=1,\ldots,\ell} \Pr\brac{\bigcap_{i=1}^\ell \cS_i^*}\\
\sim& \frac{1}{\ell!}\sum_{1\leq s_1,\ldots,s_\ell\leq\om}{ \binom{n}{s_1,\ldots,s_\ell, n-s}} \prod_{i=1}^\ell \bfrac{2s_i}{n}^{s_i}e^{-2s_i}\k_{s_i}\sim \frac{1}{\ell!}\sum_{1\leq s_1,\ldots s_\ell\leq\om}\prod_{i=1}^\ell \frac{(2s_i)^{s_i}}{s_i!}e^{-2s_i}\k_{s_i}\\
\sim&\frac{1}{\ell!}\brac{\sum_{s=1}^\infty \frac{(2e^{-2})^{s}}{s}\s_{s}}^\ell\sim \frac{\f_R^\ell}{\ell!}.
\end{align*}
The first approximation follows from Lemma \ref{smallcond} and the second from \eqref{EXL}, \eqref{kappal}.

Using Inclusion-Exclusion, we have
\[
\Pr(\S_\s\mid\cB)=\sum_{\ell\geq \s}(-1)^{ \ell-\s}\binom{\ell}{\s}T_\ell\sim \sum_{\ell\geq \s}(-1)^{\ell-\s}\binom{\ell}{\s}\frac{\f_R^\ell}{\ell!}=\frac{\f_R^\s e^{-\f_R}}{\s!}.
\]
{ Lemma \ref{smallsmall}} gives  the unconditional probability.
\end{proof}
Let $\bX_k$ count the number of simple $k$-sequences as in Lemma \ref{EXk}.
\begin{lemma}\label{X|d}
If $\s=O(1)$ then $\E(\bX_k\mid \S_\s)\sim 1$.
\end{lemma}
\begin{proof}
\begin{align*}
\E(\bX_k\mid \S_\s)&=\sum_{\cB=(B_1,\ldots,B_k)}\Pr(\cB\mid\S_\s)\\
&=\sum_{\cB=(B_1,\ldots,B_k)}\frac{\Pr(\S_\s\mid\cB)\Pr(\cB)}{\Pr(\S_\s)}\\
&=\sum_{\cB=(B_1,\ldots,B_k)}\frac{\Pr(\cB)}{\Pr(\S_\s)}\sum_{\ell\geq \s}(-1)^{\ell-\s}{ \binom{\ell}{\s}}T_\ell \\
&=\sum_{\cB=(B_1,\ldots,B_k)}\frac{\Pr(\cB)}{\Pr(\S_\s)}\sum_{\ell\geq \s}(-1)^{\ell-\s}{ \binom{\ell}{\s}}\frac{1}{\ell!}\sum_{1\leq s_1,\ldots,s_\ell\leq\om}\sum_{|S_i|=s_i,\atop i=1,\ldots,\ell} \Pr\brac{\bigcap_{i=1}^\ell \cS_i^*\bigg|\cB}\\
&\sim\sum_{\cB=(B_1,\ldots,B_k)}\frac{\Pr(\cB)}{\Pr(\S_\s)}\sum_{\ell\geq \s}(-1)^{\ell-\s}{ \binom{\ell}{\s}}\frac{1}{\ell!}\sum_{1\leq s_1,\ldots,s_\ell\leq\om}\sum_{|S_i|=s_i,\atop i=1,  \ldots,\ell} \Pr\brac{\bigcap_{i=1}^\ell \cS_i^*}\\
&\sim\sum_{\cB=(B_1,\ldots,B_k)}\frac{\Pr(\cB)}{\Pr(\S_\s)}\Pr(\S_\s)\\
&=\E(\bX_k)\sim 1.
\end{align*}
\end{proof}

\section{Joint distribution of small and large dependencies}\label{Jointd}
\subsection{$P_n(0,d)$: the case of no small fundamental dependencies.}
Let $P_n(0,d)$ be the probability that $M \in \ul M(n)$ has no small fundamental dependencies and the maximum number of large simple dependencies is $d$. Let $\pi(d)$ be given by \eqref{ldef}.
The purpose of this section is to prove the following.
%%%%
\beq{S0Lk}{
P_n(0,d)\sim \pi(d)\;e^{-\f}.
}

Let $V$ be the vector space generated by the dependencies.
 Let $\cL_\l$ be the event that the dimension of $V$ is $\l$. Let
\[
p(0,\l)=\Pr(\S_{0}\wedge \cL_{\l})\text{ and } p(0)=\Pr(\S_0).
\]

\begin{lemma}\label{dimd}
For $0\leq \l=O(1)$, $p(0,\l)\sim P(0,\l)$ where $P(0,\l)=\pi(\l)\;e^{-\f}$.
\end{lemma}

\begin{proof}
For $0\leq k=O(1)$, we have from Lemma \ref{X|d} that
\beq{42.5}{
1\sim \E(\bX_k\mid \S_0)=\sum_{\l\ge k}
 \E(\bX_k\mid \S_0\wedge \cL_{\l})\times \frac{p(0,\l)}{p(0)}.
}
Let $\cH$ be the event that there exists a set of dependent rows $H$ where $\om\leq |H|\notin J_1$.
Then we have
\begin{align}
\E(\bX_k\mid \S_0\wedge \cL_{\l})&=\E(\bX_k\mid \S_{0},\wedge \cL_{\l}\wedge\neg\cH)\Pr(\neg\cH)+ \E(\bX_k\mid \S_0\wedge \cL_{\l}\wedge\cH)\Pr(\cH)\nn\\
&\sim\prod_{i=0}^{k-1}(2^{\l}-2^{i}).\label{42.6}
\end{align}
{\bf Justification for \eqref{42.6}:}
Given $\S_0\wedge \cL_{\l}\wedge \neg\cH$ there are $2^{\l}$ vectors in $V$. Choosing $i$ members of a simple sequence generates a subspace of dimension $i$, and we eliminate $2^{i}$ vectors from consideration as the next member of the sequence. Given $\neg\cH$  the number of simple sequences is given by the RHS of \eqref{42.6}.
%Furthermore, this product is an upper bound on simple sequences. Given %$B_1,B_2,\ldots,B_i$ the $2^{i}$ vectors $\bigoplus_{j=1}^i B_j^{x_j}$ are %distinct and cannot be the next member of the sequence.
Equation \eqref{42.6} then follows from $\Pr(\cH)=o(1)$.

It follows from \eqref{42.6} that for $\l\geq 0$,
\beq{42.7}{
1\sim \sum_{\l=k}^\infty \frac{p(0,\l)}{p(0)}\prod_{i=0}^{k-1}(2^{\l}-2^i).
}
The asymptotic solution of \eqref{42.7}  is given by the following lemma.
\begin{lemma}
For $\l \ge 0$, the solutions  to
\begin{align}
1&=\sum_{\l=k}^\infty q_{\l}\prod_{i=0}^{k-1}(2^{\l}-2^{i}),\qquad k\geq 0.\label{qd}
\end{align}
are given by $q_\l=\pi(\l)$ of \eqref{ldef}.
\end{lemma}
\begin{proof}
Gaussian coefficients are defined as
\beq{shit}{
\genfrac{[}{]}{0pt}{}{\l}{k}_z =
\frac{\prod_{i=1}^k (z^{\l-i+1}-1)}{\prod_{i=1}^k (z^{i}-1)}.
}
Using \eqref{shit} with $z=2$, equation \eqref{qd} can be rewritten as
\begin{align}\label{dshit}
1
&= 2^{\binom{k}{2}} \prod_{i=1}^{k}(2^{i}-1)\; \sum_{\l=k}^\infty q_{\l}
\genfrac{[}{]}{0pt}{}{\l}{k}_2.
\end{align}
Put $\psi_{k}=1/\brac{2^{\binom{k}{2}}\prod_{i=1}^{k}(2^{i}-1)}$, we see that $q_{\l}$ is the solution to
\beq{qd1}{
\sum_{\l=k}^\infty \genfrac{[}{]}{0pt}{}{\l}{k}_2q_{\l}=\psi_{k},\qquad k\geq 0.
}
Fix $\d\geq 0$,  multiply equation $k\geq \d$ in \eqref{qd1} by $(-1)^{k-\d}2^{\binom{k-\d}{2}}\genfrac{[}{]}{0pt}{}{k}{\d}_2$, and  sum these equations over $k\geq \d$. This gives
\begin{align}
\sum_{k=\d}^\infty(-1)^{k-\d}2^{\binom{k-\d}{2}}
\genfrac{[}{]}{0pt}{}{k}{\d}_2\psi_{k}
& =\sum_{k=\d}^\infty\sum_{\l=k}^\infty (-1)^{k-\d}\genfrac{[}{]}{0pt}{}{k}{\d}_22^{\binom{k-\d}{2}} \genfrac{[}{]}{0pt}{}{\l}{k}_2q_{\l}\label{g0}\\
&=\sum_{k=\d}^\infty\sum_{\l=k}^\infty (-1)^{k-\d}\genfrac{[}{]}{0pt}{}{\l-\d}{k-\d}_22^{\binom{k-\d}{2}} \genfrac{[}{]}{0pt}{}{\l}{\d}_2q_{\l}\nn\\
&=\sum_{\l=\d}^\infty\genfrac{[}{]}{0pt}{}{\l}{\d}_2q_{\l} \sum_{k=\d}^\l (-1)^{k-\d}\genfrac{[}{]}{0pt}{}{\l-\d}{k-\d}_22^{\binom{k-\d}{2}} \label{g1}\\
&=q_{\d}.\label{g2}
\end{align}
{\bf Explanation: \eqref{g1} to \eqref{g2}:}
Gaussian coefficients satisfy the identity
\beq{Gausse}{
(1+x)(1+zx)\cdots(1+z^{r-1}x)=\sum_{\ell=0}^r\genfrac{[}{]}{0pt}{}{r}{\ell}_z z^{\binom{\ell}{2}}x^\ell.
}
To prove the last summation on the right hand side of \eqref{g1} is zero for
$\l > \d$, use  \eqref{Gausse} with $x=-1,z=2$, $\ell=k-\d$  and $r=\l-\d$.
This gives $\sum_{\ell=0}^{\l-\d}\genfrac{[}{]}{0pt}{}{\l-\d}{\ell}_2 2^{\binom{\ell}{2}}(-1)^\ell=0$ for $\l > \d$.

For $z<1$, taking the limit of \eqref{Gausse}  gives
\begin{equation}\label{iqbin}
\prod_{\ell=0}^{\infty}(1+z^\ell x) = \sum_{\ell=0}^{\infty}
\frac{z^{\ell \choose 2} x^\ell}{ \prod_{i=1}^\ell (1-z^i)} .
\end{equation}
Replacing $\d$ by $\l$ in equation \eqref{g0}, we see that the solution $q_\l$ to \eqref{qd} is
\begin{flalign}
q_{\l}&=\sum_{k=\l}^\infty\frac{(-1)^{k-\l}2^{\binom{k-\l}{2}-\binom{k}{2}}} {\prod_{i=0}^{\l-1}(2^{\l-i}-1)\prod_{i=\l}^{k-1}(2^{k-i}-1)} \nn\\
&=\frac{\bfrac{1}{2}^{\l^2}}{\prod_{i=1}^\l\brac{1-\bfrac{1}{2}^i}}
\sum_{\ell=0}^\infty \frac{ (-1)^\ell \bfrac{1}{2}^{\ell \choose 2}\bfrac{1}{2}^{(1+\l)\ell}}{\prod_{i=1}^\ell\brac{1-\bfrac{1}{2}^i}}\label{wot1}\\
&=\bfrac{1}{2}^{\l^2}\frac{\prod_{i=\l+1}^\infty\brac{1-\bfrac{1}{2}^i}}
{\prod_{i=1}^\l\brac{1-\bfrac{1}{2}^i}}= \pi(\l),\label{wot2}
\end{flalign}
where $\pi(\l)$ is given in  \eqref{ldef}. To get from \eqref{wot1} to \eqref{wot2}, use \eqref{iqbin} with $z=1/2$ and $x=(-1/2^{\l+1})$.
\end{proof}

The $p(0,\l)$ only satisfy \eqref{qd} asymptotically and so to prove the lemma, we show that for large $K$,
\beq{qeps}{
\sum_{\substack{\l\geq K\\\s\geq 0}}q_{\l}\leq \e,
}
 where $\e>0$ is arbitrarily small. Now,
\[
\prod_{i=0}^{k-1}(2^\l-2^i)=2^{k\l}\prod_{i=0}^{k-1}\brac{1-\frac{1}{2^{\l-i}}}\geq 2^{k\l}\brac{ 1-\sum_{i=0}^{k-1}\frac{1}{2^{\l-i}}}\geq 2^{(k-1)\l}.
\]
It follows that
\[
\sum_{\substack{\l\geq K\\\s\geq 0}}q_{\l}\leq 2^{-K(K-1)}.
\]
Thus \eqref{qeps} holds if $K\geq \sqrt{2\log_21/\e}$.
\end{proof}

\subsection{$P_n(1,d)$: the case of one small fundamental dependency.}
Introduce the notation $P_n([m,d])$ for the probability that $M \in \ul M(n)$ has exactly $m$ small fundamental dependencies and the maximum number of large simple dependencies is $d$. Thus there are small dependencies  $D_1,...,D_m$  and (not necessarily unique) large dependencies $B_1,...,B_d$ corresponding to $M$ having a null space of dimension $m+d$. In the case $m=0$, it follows from \eqref{S0Lk} that $P_n(0,d)\sim \pi(d)\;e^{-\f}$.

Before considering $P_n(m,d)$, we explain the basic principle by deriving $P_n(1,d)$. The general case will follow from the recursive application of this.

Let $M \in \ul M([n])$ and let $L$ be a fixed set of rows, $|L|=\ell$. We write
\[
M=\brac{\begin{array}{cc}
S_L&R\\
C&M'
\end{array}}.
\]
Here
 $S_L$ is $\ell \times \ell$, $R$ is $\ell \times (n-\ell)$, $C$ is $(n-\ell) \times \ell$
and $M'$ has rows and columns indexed by $[n]-L$.

The event $R=\ul 0$, is dependent only on the columns of $[n]-L$ in $M$.
Provided $\ell=o(n^{1/2})$, $R=\ul 0$ has probability
\[
\Pr(R=\ul 0)=\brac{1-\frac{\ell}{n}}^{2(n-\ell)}\sim e^{-2 \ell}.
\]
Given $R=\ul 0$, $M'$ is a uar element of $\ul M([n]-L)$. This follows directly from the fact that $M$ is a uar element of $\ul M([n])$. Each column of $M$ has 2 random entries, and these are not in the rows of $L$. At this point
\begin{equation}\label{GoodM}
M=\brac{\begin{array}{cc}
S_L&0\\
C&M'
\end{array}}.
\end{equation}
The event $\cD_L$ that within the columns of $L$ the sub-matrix $S_L$ is the vertex-edge incidence matrix of a connected random mapping $D_L$ is independent of what happens in the columns of $[n]-L$ in $M$.  Each column of the sub-matrix $C$  has one uar entry, is (see Section 2) and we have
 \[
 P(\cD_L)\sim
  \bfrac{2\ell}{n}^\ell \cdot\frac{(\ell-1)!}{\ell^\ell} \sigma_\ell .
 \]
The probability { $P_{n-\ell}(0,k)\sim e^{-\f}\pi(k)$} that $M'$ has no small dependencies and $k$ large ones  is given by \eqref{S0Lk} above.
Let $P^*(j,k;1)$ be the probability that exactly $j$ of the $k$ large dependencies of $M'$ remain as dependencies after adding back the sub-matrix $C$.  To maintain continuity of exposition, the analysis of this event is deferred until Section \ref{goback}. Equation \eqref{JKM}  of Section \ref{goback} with $m=1$, gives
\[
P^*(j,j;1)= \bfrac{1}{2}^j\text{ and }P^*(j,j+1;1)=1- \bfrac{1}{2}^{j+1}.
\]
Let
\[
P_n(1,j,L)=
\Pr(M \text{ has 1 small  fundamental dependency $L$ and $j$ large dependencies}),
\]
Thus using \eqref{S0Lk}, and the above
\beq{PnL}{
P_n(1,j,L)\sim
\bfrac{2}{n}^\ell (\ell-1)! \sigma_\ell \cdot e^{-2 \ell}\cdot e^{-\f} \;
\brac{\pi(j)\bfrac{1}{2}^j+ \pi(j+1)\brac{1- \bfrac{1}{2}^{j+1}}}.
}
The probability that $L$ is dependent, but $R \ne \ul 0$ is $O(\ell^2/n)$.
The events that $L$ is the unique fundamental dependency are exclusive and exhaustive, so $\Pr([1,j])$ is  the sum of these. Thus, summing \eqref{PnL} over $L$ for $L\neq \emptyset$ gives
\begin{align*}
P_n(1,j)\sim
\f\, e^{-\f}\cdot \brac{\pi(j)\bfrac{1}{2}^j+ \pi(j+1)\brac{1- \bfrac{1}{2}^{j+1}}}.
\end{align*}

\subsection{The general case of $\nul(M)=d$, with $m$ small fundamental dependencies}
%}

The matrix $M'$ in \eqref{GoodM} is a uar element of $\ul M([n]-L)$, and we can repeat the above construction with $M'$ instead of $M$. We remove a set of columns $L'$ and conditional on $R'=\ul 0$, the sub-matrix $M''$ is a uar element of
$\ul M([n]-L-L')$.  In this way we can obtain the probability $\Pr([2,j])$ of two small and $j$ large dependencies, and so on.

To systematize this, let $M_0=M, L_0=L, R_0=R, n_0=n, \ell_0=|L|$ and let $M_1=M', n_1=n_0-\ell_0$. Thus, $M_0$ is a uar element of $\ul M(n_0)$ and with some relabelling of $[n]-L$, $M_1$ is a uar element of $\ul M(n_1)$, etc.

In this way, we remove a sequence $(L_0,L_1,...,L_{m-1})$ of column sets, of total size at most $m \om$. As $n-m \om \sim n$, equation \eqref{S0Lk} holds in $\ul M(n_{m})$ with the same asymptotic probability. Taking the subspace $\ul M([0,k], n_{m})$ of $M(n_{m})$, we work back to the subspace of $\ul M$ with
small fundamental dependencies $L_0,...,L_m$ and $j \le k$ large dependencies, and thus to $P_n(m,j])$, the probability of $\ul M([m,j],n)$.

Summarizing, we have
\begin{flalign}
\Pr(R(L_j)=\ul 0, \; j=0,...,m-1) &\sim \prod_{j=0}^{m-1} e^{-2 \ell_j}, \label{Pox1}
\\
P(\cD_{L_j}, j=0,...,m-1) &\sim \prod_{j=0}^{m-1}
 \bfrac{2\ell_j}{n_j}^{\ell_j} \cdot\frac{(\ell_j-1)!}{\ell_j^{\ell_j} } \sigma_{\ell_j},\label{Pox2}\\
P^*(j,j+r;m)&\sim  \qchoose{m}{r}{2}\bfrac{1}{2}^{(j+r)(m-r)} \prod_{j=h+1}^{h+r} \brac{1-\bfrac{1}{2}^j},
\label{Pox3}\\
{ P_{n-\ell}(0,k)} & \sim e^{-\f}\;\pi(k). \label{pef}
\end{flalign}
The last line is \eqref{S0Lk}. For continuity of exposition, the proof of \eqref{Pox3} is deferred until Theorem \ref{jofk} in Section \ref{goback} below.

The dependency of probability  in \eqref{Pox3} on the sizes $\ell_j \le \om, j=0,...,m-1$,
is hidden in the $(1+o(1))$ term in the asymptotic notation. We multiply \eqref{Pox1} by \eqref{Pox2}, and sum over all distinct sets of removed columns $(L_0,...,L_{m-1})$, and noting that each entry is repeated $m$ times in such sequences, we obtain a quantity $\Psi(m)$ given  by
\begin{flalign*}
\Psi(m) & \sim  \frac{1}{m!}\sum_{\ell \ge 1} \sum_{\ell=\ell_0+\cdots+\ell_{m-1}}\;
{n \choose \ell_0,\ldots,\ell_{m-1}}\;\prod_{j=0}^{m-1}
\brac{\Pr(R(L_j)=\ul 0)\cdot P(\cD_{L_j})}\\
&\sim \frac{1}{m!}\sum_{\ell \ge 1} \sum_{\ell=\ell_0+\cdots+\ell_{m-1}}
\prod_{j=0}^m (2e^{-2})^{\ell_j} \frac{1}{\ell_j} \s_{\ell_j} \\
&= \frac{\f^m}{m!}.
\end{flalign*}
Thus, multiplying $\Psi(m)$ by \eqref{Pox3} and \eqref{pef}, and summing over $k \ge j$ large dependencies,
\begin{flalign}
P_n(m,j) &\sim \frac{\f^m}{m!}\, e^{-\f}\;\sum_{r=0}^m \pi(j+r) P^*(j,j+r;m)  \nonumber\\
&\sim \frac{\f^m}{m!}\,e^{-\f} \;
\sum_{r=0}^m \pi(j+r) \qchoose{m}{r}{2}\bfrac{1}{2}^{(j+r)(m-r)} \prod_{j=h+1}^{h+r} \brac{1-\bfrac{1}{2}^j}.
\label{OKorNOT}
\end{flalign}
Finally, the probability that $\nul(M)=d$ is
\[
\Pr(\nul(M)=d)=\sum_{m=0}^d P_n(m,d-m),
\]
which  completes the proof of Theorem \ref{TH1}.

\subsection{Going back from $M'$ to $M$. Change in dimension of null space.}\label{goback}
Write $M=\brac{\begin{array}{cc}
S_L&0\\
C&M'
\end{array}}$ as given in \eqref{GoodM}.
In this section we prove the following theorem.
\begin{theorem}\label{jofk}
Suppose that the $(n-L) \times (n-L)$ sub-matrix $M'$ of $M$
has no small dependencies, and $k$ large simple dependencies,
and the $L \times L$ sub-matrix $S_L$ of $M$
has $m$ small fundamental dependencies of total size $L$.
For $k=h+r$, where $0 \le r \le m$,
 the probability %of the event ${\cal E}(h,k;m)$
the maximum number of large simple dependencies in $M$ is $h$, is asymptotic to
\begin{equation}\label{JKM}
P^*(h,h+r;m)= \qchoose{m}{r}{2} \bfrac{1}{2}^{(h+r)(m-r)} \prod_{j=h+1}^{h+r} \brac{1-\bfrac{1}{2}^j}.
\end{equation}
\end{theorem}

Before proceeding with the proof of Theorem \ref{jofk}, we give an outline of the proof structure. Each  column of the sub-matrix $C$ has a unique random non-zero entry in the rows of $M'$. On average about $\ell/2$ of these non-zeros  fall in the rows of any large dependency $B$ of $M'$. To extend $B$  to a dependency $A$ of $M$, we may  need to include some   rows of $S_L$ in $A$ to cancel any non-zeros of $C$ which fall in the rows of $B$.

Thus in general $A \cap L \ne \es$, and some rows of $A$ have been deleted to give $B$. If $M'$ has $k$ large dependencies $B_1,...,B_k$, then any extension of these sets needs to preserve and extend the intersection structure $I'_\ux, \; x \in \{0,1\}^k$ in $M'$ to $M$. If $j \le k$ of the sets $B_i$ extend successfully then the final intersection structure will be given by $I_y, y\in \{0,1\}^j$. The interaction of this structure with $L$ is the one described in Section \ref{ESL} and summarized by \eqref{crispy}. The extensions are not unique. If $A$ is a large dependency, and $L$ is small, then $A \D L$ is large. It was exactly this problem which obliged us to construct our proofs in this way.

\subsubsection*{Proof of Theorem \ref{jofk}}
Suppose $M'$ has $k$  large dependencies $B_1,...,B_k$ but no small dependencies.  In this case there is
a well defined  vector space of dimension $k$ spanned by $B_1,...,B_k$. Assume the $m$ small dependencies $D_j, j=0,...,m-1$ occupy the first $L$ columns. The matrix $M$ can be written as follows.
\[
M=\left(\begin{array}{cccccc}
D_0&0&0&\cdots&0&0\\
C_{0,1}&D_1&0&\cdots&0&0\\
C_{0,2}&C_{1,2}&D_2&\cdots&0&0\\
\vdots&\vdots&\vdots&\ddots&0&0\\
C_{0,m-1}&C_{1,m-1}&C_{2,m-1}&\cdots&D_{m-1}&0\\
C_{0,m}&C_{1,m}&C_{2,m}&\cdots&C_{m-1,m}&M'
\end{array}
\right) .
\]

Let $|D_j|=\ell_j$ where $L=(\ell_0+\cdots+\ell_{m-1})$, and $n_j=n-(\ell_0+\cdots+\ell_{j-1})$.
Each $(n_j -\ell_j) \times \ell_j$ sub-matrix $C_j=(C_{j,j+1},...,C_{j,m})^\top$ has exactly one random one in each column.
The probability any of these ones fall in any $C_{j,i}$ where $j+1 \le i \le m-1$ for $j=0,...,m-1$ is $O(\om^3/n)$. Conditional on this not occurring, the non-zero entry in each column is u.a.r. in $n'=n-L$. Tidying up, and writing $C'_j=C_{j,m}$ we have
\beq{diag}{
M=\left(\begin{array}{cccccc}
D_0&0&0&\cdots&0&0\\
0&D_1&0&\cdots&0&0\\
%C_{0,2}&C_{1,2}&D_2&\cdots&0&0\\
\vdots&\vdots&\ddots&0&0\\
0&0&0&\cdots&D_{m-1}&0\\
C_{0}'&C_{1}'&C_{2}'&\cdots&C'_{m-1}&M'
\end{array}
\right)  \quad =\quad
\left(\begin{array}{cc}
D&0\\
C&M'\end{array}
\right).
}
Assuming the above structure for $M$, write $B_j\diamond D_s$, and say the rows $B_j$ agree with $D_s$, if there exists a set of row indices $J_{j,s}$, a subset of the row indices of $D_s$, such that the rows $B_j\cup J_{j,s}$ are zero sum in the columns  of $D_s$. %,...,D_{m-1}, M'$.
Otherwise we say $B_j$ is inconsistent on $D_s$, as $B_j$ cannot be extended to a large dependency in $M$.

%We will be able to extend $B_r$ to a dependency %$A_r=B_r\cup\brac{\bigcup_{s=1}^mJ_{r,s}}$ if and only if $B_r\diamond %D_s,s=1,2,\ldots,m$. In which case we say  $\cB_r$ occurs.

For $i \in D_s$, column $i$ has a unit entry in row $i$, and if the random unit entries are in rows $t,t'$,
we   use the notation $e_1(i)=t\in D_s,\, e_2(i)=t'\notin D_s$. Let $H_s$ be the set of column indices associated with the vertices of the cycle in $D_s$.

Given a maximal simple sequence $(B_1,...,B_k)$, let $\ul Q=\ul Q(k)$ be the $k$-dimensional vector space generated by the first $k$ rows of the $K \times K$ matrix $U$, the rows corresponding to $B_1,...,B_k$; see Section \ref{back}. For a given $D$ with cycle vertices $H$,  let $x_i^{(j)}=\indi{e_2(i) \in R(I_j),\; i \in H}$ be the indicator that $e_2(i)$ of vertex $i$ falls in the rows of the  index set $I_j,\; j=1,...,K=2^k-1$ obtained from the dissection of $(B_1,...,B_k)$ in $M'$. Let $p_j=\sum_{i \in H} x_i^{(j)}$ and $\ul p$ the $K$-vector of parities of the index set rows.

Let $T=\{\ul y:U\ul y=\ul 0\}$ be the set of parity vectors which agree with all of $B_1,...,B_k$, and $S=\{\ul w \in \ul Q: \ul w \cdot \ul p=0\}$ be the rows of $U$ which agree with a given parity vector $\ul p$.

Depending on $\ul p$, the dimension of $\ul Q$ is either reduced by zero or one by the small dependency $D$.  The set $D$ agrees with $(B_1,...,B_k)$ iff  $\ul p \in T$.

\begin{lemma}\label{lemma1}\label{lemsip}

\begin{enumerate}[(a)]
\item $B_j\diamond D$ if and only if $|\set{i\in H:\;e_2(i)\in B_j}|$ is even.

\item If $\ul p \in T$ then $S= \ul Q$ and  $(B_1,...,B_k)$ agree with $D$.\\
If $\ul p \notin T$ then $|S|= |\ul Q|/2$, and there is a basis of $S$ of dimension $k-1$ corresponding to a maximal simple sequence $(B_1',...,B_{k-1}')$ which agrees with $D$.
%
%\item Over the random choices, $e_2(i),i\in D_s,s=1,2,\ldots,m$,
%\[\Pr(B_r\diamond D_s,s=1,2,\ldots,m)\sim 1/2^m.\]
\item Let $\cB_j=B_j \diamond D$. Suppose that $Y\subseteq [j]$ is arbitrary. Then
    $$\Pr(\cB_{j+1}\mid \cB_i,i\in Y,\neg\cB_i,i\notin Y)\sim \Pr(\cB_{j+1})\sim 1/2.$$
Thus the occurrence of $\cB_{j+1}$ is asymptotically independent of the occurrence or non-occurrence of the events $\cB_1,\cB_2,\ldots,\cB_j$.
It follows that $\Pr( \ul p \in T)=\Pr(\cB_1\cdots \cB_k) \sim 1/{2^k}$.
\end{enumerate}
\end{lemma}

\begin{proof}
(a) Suppose the vertices of the cycle of $D=D_s$ are labelled $1,...,\ell$, with edges $(1,2),...,(\ell-1,\ell),(\ell,1)$. Let $(i,i+1)$ be such an edge, where $i,i+1 \in D_s$ and thus $i+1=e_1(i)$. Then let $x_i=1$ if $e_2(i)\in B_j$. We introduce variables $y_i,z_i,i=1,2,\ldots,\ell$, which will be used to define the index set of rows $J_{j,s}$, if this is possible. We interpret $y_i=1$ to mean $i\in J_{r,s}$ and $z_i=1$ to mean that $e_1(i)\in J_{r,s}$. For $B_j\cup J_{j,s}$ to be a dependency we need  $x_i+y_i+z_i=0$ for $i=1,2,\ldots,\ell$. For consistency we need $y_{i+1}=z_{i}$ for $i=1,2,\ldots,\ell$ where $y_{\ell+1}=y_1$. This leads to the equations $y_i+y_{i+1}=x_i,i=1,2,\ldots,\ell$.
These equations are feasible if and only if
\beq{x1x2}{
x_1+x_2+\cdots+x_\ell=0.
}
If \eqref{x1x2}  holds there are exactly two possible choices for the $y_i$. Choosing  an arbitrary value in $\{0,1\}$ for $y_1$, determines $y_i, i=2,...,\ell$ and thus $J_{j,s}=\set{i:y_i=1}$.

\begin{figure}[H]
\begin{center}
\begin{tikzpicture}[scale=2.5]
\draw [thick, dashed] (0.1,0)--(0.9,0) (2.1,0)--(2.9,0);
\draw [thick, dashed] (0.1,1)--(0.9,1) (2.1,1)--(2.9,1);
\draw  (1,1) circle (0.08); \draw  (2,1) circle (0.08);
\draw  (1,0) circle (0.08); \draw  (2,0) circle (0.08);
\draw  (0,0) circle (0.08); \draw  (3,0) circle (0.08);
\draw (0,1) circle (0.08); \draw (3,1) circle (0.08);
\draw [thick,->] (1,0.1)--(1,0.9);
 \draw [thick,->]  (1.1,1)--(1.9,1);\draw [thick,->]  (2,0.9)--(2,0.1); \draw [thick,->]  (1.9,0)--(1.1,0);

\node [below] at (0,-0.1) {$x_1$};
\node [below] at (0.9,-0.1) {$y_1=z_4$};
\node [below] at (2.1,-0.1) {$y_4=z_3$};
\node [below] at (3,-0.1) {$x_4$};

\node [above] at (0,1.1) {$x_2$};
\node [above] at (0.9,1.1) {$y_2=z_1$};
\node [above] at (2.1,1.1) {$y_3=z_2$};
\node [above] at (3,1.1) {$x_3$};
\node at (1,0) {$\scriptstyle 1$};
\node at (1,1) {$\scriptstyle 2$};
\node at (2,1) {$\scriptstyle 3$};
\node at (2,0) {$\scriptstyle 4$};
\end{tikzpicture}
\caption{\;\;Example: Cycle $(1,2,3,4)$ with labelling. The edges $(i,e_1(i))$ are drawn solid, and edges $(i,e_2(i))$ dashed. } \label{cyclefig}
\end{center}
\end{figure}
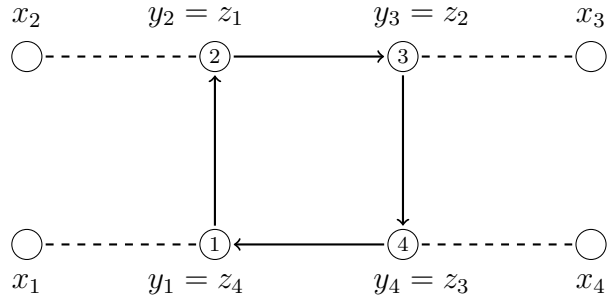
%\vspace{-0.4in}

We deal with the attached trees by working backwards from the cycle to the leaves. Suppose that vertex $i$ is not in the cycle and that the values $x_j,y_j,z_j$ have been determined for its parent $j$ its tree. We are forced to take $z_i=y_j$ and then $y_i$ is determined from $x_i+y_i+z_i=0$. Each time we find that $y_i=1$, we add $i$ to $J_{r,s}$.

(b)Let $\ul z_1,\ul z_2 \in \ul Q\sm S$ then $(\ul z_1+\ul z_2) \cdot \ul p=0$ so
$\ul z_1+S=\ul z_2+S$; the subgroup $S$ has only one non-trivial coset in the group $\ul Q$. Thus $|S|=|\ul Q|/2$, the dimension of $S$ is $k'=k-1$, and some $k'$ rows of  $S$ form a
basis for the reduced matrix $U'=S$.

(c) Let ${\cal P}=(P_0,P_1,P_2,\ldots,P_{2^j-1})$ be the partition of $[n-L]$ induced by $B_1,B_2,\ldots,B_j$.
% as described in Section \ref{ESL}. "cant find this" CDC
Each part of the partition contains $\sim n/2^j$ rows. The occurrence of $\cB_i,i\in Y,\neg\cB_i,i\notin Y$ is determined by the the allocation of the $e_2(i)$ into each part. As such, if $e_2(i)$ lies in some part $P_t,\,1\leq t\leq 2^j$, then it is distributed uniformly over $P_t$.
Each part of ${\cal P}$ corresponds to an index $\ul w \in \{0,1\}^j$. The introduction of $B_{j+1}$ splits each $P_t$ into two  parts with index sets $(\ul w,0),\;(\ul w,1)$ of asymptotically equal size.
If $e_2(i)$ lies in $P_t$, in one ``half'' we will have $x_i=0$ and in the other ``half'' we will have $x_i=1$, where $x_i$ is computed with respect to $B_{j+1}$.    It follows that \eqref{x1x2} holds with probability $\sim1/2$.

\end{proof}

\ignore{%%%%%%%%
(b) This follows from the fact that $|B_i|\sim n/2$ and $|L|=o(n^{1/2})$ and $e_2(i)$ is chosen randomly from a set of size $n-\om$. Furthermore, the values $e_2(i),i\in D_{s_1}$ are independent of the choices
$e_2(i),i\in D_{s_2}$ if $s_1\neq s_2$.

(c) Let ${\cal P}=(P_1,P_2,\ldots,P_{2^j})$ be the partition of $[n-L]$ induced by $B_1,B_2,\ldots,B_j$ as described in Section \ref{ESL}. Each part of the partition contains $\sim n/2^j$ rows. The occurrence of $\cB_i,i\in Y,\neg\cB_i,i\notin Y$ is determined by the the allocation of the $e_2(i)$ into each part. As such, if $e_2(i)$ lies in some part $P_t,\,1\leq t\leq 2^j$ then it is distributed uniformly over $P_t$. Each part of ${\cal P}$ corresponds to an index $\ux \in \{0,1\}^j$ and is split into two  parts with indices $(\ux,0),\;(\ux,1)$ by $B_{j+1}$. These parts are  of asymptotically equal size. In one ``half'' we will have $x_i=0$ and in the other ``half'' we will have $x_i=1$, where $x_i$ is computed with respect to $B_{r+1}$.    It follows that \eqref{x1x2} holds with probability $\sim1/2$.
\end{proof}
}%%%%%%%%%%

\subsubsection*{Proof of Theorem \ref{jofk}}

For convenience, let $k=h+r$, then by Lemma \ref{lemsip}(b), $0 \le r \le m$. For a small dependency $D_i$, $i=1,...,m$, let $s_i=1$ if at least one of the remaining $1 \le k' \le k$ simple large dependencies $(B'_1,...,B'_k)$ is inconsistent on $D_i$, thus reducing $k'$ to $k'-1$; and $s_i=0$ otherwise. By Lemma \ref{lemsip}(c),
$\Pr(s_i=0) \sim 1/2^{k'}$.

Let $S_r=\{s \in \{0,1\}^m: \sum_{i=1}^m s_i=r\}$ be those sequences $s$ with $r$ unit entries. Let $d(s)=(d_0,d_1,...,d_r)$ where $d_j$ is the number of zeroes between the $j$--th and $j+1$--th unit entry of $s$; and thus $\sum d_j=m-r$.
The probability of a given sequence $s$ is asymptotic to $\r_r(s)$ where
\begin{align*}
\r_r(s) =& \;\prod_{j=0}^r \bfrac{1}{2^{k-j}}^{d_j}\; \prod_{j=0}^{r-1} \brac{1-\frac{1}{2^{k-j}}} \\
=& \;2^{\sum_{j=0}^r j d_j} \cdot \bfrac{1}{2^k}^{(m-r)} \prod_{j=0}^{r-1} \brac{1-\frac{1}{2^{k-j}}}
.
\end{align*}
To obtain $P^*(h,h+r;m)$, we need to sum $\r_r(s)$ over $s \in S_r$.

The polynomial ${\scriptstyle \qchoose{m}{r}{q}}=\frac{(q^m-1)...(q^{m-r+1})}{(q^r-1)...(q-1)}$,  is the number of $r$-dimensional subspaces of $m$-dimensional space over $GF(q)$, and thus enumerates
the number of $r\times m$ matrices over $GF(q)$ with no zero rows which are in reduced echelon form. As a consequence of this, it is also the generating function for the total number  of {\em inversions}, $i(s)$, in  sequences $s \in S_r$, (see \cite{Sach}, Chapter 3.4.5). An inversion in a $0-1$ sequence $s$, is a pair $(1,0)$ contained in $s$, and  $i(s)=\sum_{j=0}^r j d_j$.
 Thus,
\begin{equation}\label{inversions}
\qchoose{m}{r}{q}
= \sum_{s \in S_r \atop d(s)=(d_0,d_1,...,d_r)}
q^{\sum_{j=0}^r j d_j}.
\end{equation}

Using  \eqref{inversions} with $q=2$, we obtain $P^*(h,h+r;m)$ as in \eqref{JKM} from,
\[
\sum_{s \in S_r} \r_r(s)= \qchoose{m}{r}{2}\;\bfrac{1}{2}^{(h+r)(m-r)} \prod_{j=h+1}^{h+r} \brac{1-\bfrac{1}{2}^j} .
\]

\section{Further comments: Rank over $GF(t)$, and $GF(2)$ for $r \ge 2, s=2,3$: Proof of Theorem \ref{th3} }

\subsection{Rank over $GF(2)$ for $r \ge 2, s=2,3$ }
\paragraph{Case $r=2,s=2$.}
An $n \times 2n$ matrix of this type has even column sum and row rank $n^*=n-1$ w.h.p.

Borrowing from \cite{FK}  Theorem 16.5, for $r=1$, the expected number of fundamental zero-sum sets of size  $\ell$ is
\[
\E X_\ell=
{n \choose \ell} \bfrac{\ell-1}{n-1}^{\ell} \bfrac{n-1-\ell}{n-1}^{n-\ell} \cdot \frac{1}{(\ell-1)^{\ell}}\sum_{k=2}^\ell (k-1)! k \ell^{\ell-k-1}\sim  e^{-\ell} \frac 1\ell \sum_{j=0}^{\ell-2} \frac{\ell^j}{(\ell)_j}.
\]
As the last sum tends to $e^{\ell}/2$ we have $\E X_\ell \le 1/\ell$.
If $L$ is zero-sum, so is $[n]-L$. For $r=2$ the total expected number of $\ell$-dependencies, $2 \le \ell \le n-2$ is at most
\[
4 \sum_{\ell=2}^{n/2} \E X_\ell \bfrac{\ell-1}{n-1}^{\ell}
\sim 4 \sum_{\ell=2}^{n/2} \bfrac{\ell-1}{n}^{\ell} \frac{1}{\ell}=O\bfrac{1}{n^2}.
\]
\paragraph{Case $r=2,s=3$.} It follows from the proofs  that an $n \times 2n$ matrix of this type has full row rank w.h.p., as the 'second matrix' cancels the constant number of dependencies in the first (if any).

\subsection{ Rank over $GF(t)$, $t>2$: Proof of Theorem \ref{TH2}}
The proof of Theorem \ref{TH2} is greatly simplified by
the w.h.p. lack of large dependencies.

\paragraph{Case I: The sum of all rows.}
Let $W(M)$ be an indicator that $\sum_{i=1}^n \ul r_i= \ul 0$, (i.e., that the rows of $M$ sum to zero). Then with arithmetic over $GF(t)$,
\[
\E W=\left\{ \begin{array}{ll}
\brac{\sum_i f_i f_{t-1-i}}^n & \text{ Model 1,\;2}\\
\brac{\sum_{i+j+k=0} f_i f_{j} f_k}^n & \text{ Model 3}
\end{array}\right. .
\]
Thus unless $t=3$ and $f_1=1$ (Model 1), $\E W \ra 0$ as $n \rai$.

\paragraph{Case II: The sum of $\ell$ rows.}
Let $L$ be a set of row indices of size $\ell$. For a given column $i$ where  $i \in L$, for the rows of $L$ to be dependent, one of two events must occur. Either there is a unique random entry in the rows of $L$ which cancels the entry $M_{i,i}$ in row $i$ (Model 2, $\g=f_{t-1}$; Model 3, $\g=\sum f_i f_{t-i}$).
Or there are 3 entries in the column which sum to zero (Model 2, $\a=\sum f_i f_{t-i-1}$; Model 3, $\a=\sum_{i+j+k=0} f_i f_{j} f_k$). For a column $i$, where $i \in [n]-L$ there must either be no random entries, or two random entries adding to zero, with probability  $\b=\sum f_i f_{t-i}$.
Thus
\begin{equation}\label{GFtq}{
\E X_\ell= {n \choose \ell} \brac{2\g \frac{\ell}{n}\bfrac{n-\ell}{n}+ \a\bfrac{\ell}{n}^2}^\ell
\brac{ \b \bfrac{\ell}{n}^2+\bfrac{n-\ell}{n}^2}^{n-\ell}.}
\end{equation}

{\bf The sum of $\ell$ rows, $\ell \le \om$.}\\
From \eqref{GFtq} above, using the methods of Section \ref{SmallLD} we find $\E Y_\ell$ is given by
\[
\E Y_\ell \sim \frac{(2\g\ell)^\ell}{\ell!} e^{-2\ell}.
\]
Extracting the moments of the fundamental dependencies $Z$ from $\E Y_\ell$
as in Section \ref{SmallLD} gives $\f_t$, as given by \eqref{fit}.

{\bf The sum of $\ell$ rows, $\om < \ell =o(n) $.}\\
As $\b, \g \le 1$ then $\sum \E X_{\ell> \om} \ra 0$. This follows by comparison with the analysis in Section \ref{LargeFM}.

{\bf The sum of $\ell$ rows, $ \ell = cn $.\;}\\
Let $\ell=cn$, then
\begin{flalign*}
\E X_{cn}=& O(1) \brac{ \frac{(2\g c(1-c)+\a c^2)^c}{c^c}\;\;
\frac{(\b c^2+(1-c)^2)^{1-c}}{(1-c)^{(1-c)}}}^n\\
=&O(1) \brac{D^c G^{1-c}}^n.
\end{flalign*}
Model 1: For $GF(3)$, $\g=0,\a=1$, and $D^c G^{1-c}=c^c (1-c)^{1-c}<1$, and thus $\E X_{cn} \ra 0$.

Model 2, 3:
We prove that, provided $1 \ge 2\g \ge \a$, then $D(c)\le 1,\;G(c)<1$ for $c \in (0,1)$, and thus $\E X_{cn} \ra 0$.

Firstly
$D(0)=2\g \le 1$, and  $D(c)=2\g-(2\g-\a)c$ which is monotone non-increasing in $c$.
Secondly $G(0)=1$, $G(1)=1$, and $G'(c)=0$ at $c=1 \pm \sqrt{\b/(\b+1)}$.
Let $\wh c= 1-\sqrt{\b/(\b+1)}$, then $G(\wh c)=2\sqrt{\b(\b+1)}-2\b$.
As $ 2\sqrt{\b(\b+1)}-2\b<1$, $G(c)$ is a minimum at $\wh c$.

\section{Appendix. Converting between the with and without replacement  models}
\label{ConV}
\subsection{$\E Y_\ell$ for $\ell$ small.}\label{ConVsm}
Regarding \eqref{EXLW}, let
\begin{flalign}\nonumber
A=& \brac{ \frac{(\ell-1)(n-\ell)}{(n-1)_2}}^\ell \brac{\bfrac{(\ell)_2}{(n-1)_2} + \bfrac{(n-1-\ell)_2}{(n-1)_2}}^{n-\ell} .
\end{flalign}
Then
\begin{flalign*}
A= & \bfrac{(\ell-1)(n-\ell)}{n^2}^\ell
\brac{\bfrac{\ell}{n}^2+\bfrac{n-\ell}{n}^2}^{n-\ell}\\
\times &
\bfrac{n^2}{(n-1)_2}^n \brac{1-\frac{3(n-\ell)+\ell-2}{\ell^2+(n-\ell)^2}}^{n-\ell}.
\end{flalign*}
However
\begin{equation}\label{e3}
\bfrac{n^2}{(n-1)_2}^n=(1+O(1/n))e^3,
\end{equation}
and for $\ell=o(n)$
\begin{equation}\label{e-3}
B=\brac{1-\frac{3(n-\ell)+\ell-2}{\ell^2+(n-\ell)^2}}^{n-\ell}=(1+O(\ell/n))e^{-3},
\end{equation}
which proves equivalence as $A \sim 1$.

\paragraph{$\E X_\ell$ for $\ell$ large.}
Note from \eqref{e-3} that $B$ is less than one for any feasible $\ell$, and if $\ell=
(n/2)(1+o(1)$ then $B=(1+O(1/n))e^{-2}$.
Also for any $\ell \rai$,
\[
(\ell-1)^{\ell}= (\ell)^{\ell} \bfrac{\ell-1}{\ell}^\ell=
 (1+O(1/\ell)) (\ell)^{\ell} e^{-1}.
\]

\paragraph{$\E (X)_k$ for $\ell \sim n/2$.}
Referring to \eqref{multinomial}, in the with-replacement model we have
\[
\Phi(\ul h, k)=
\prod_{x \ne 0}
\brac{2 \sum_{\{u,v\} \atop u+v=x}\frac {h_\uu}{n} \frac{h_\uv}{n}}^{h_\ux}
\brac{\sum_\uu \bfrac{h_\uu}{n}^2}^{h_0}
\]
The equivalent to $\Phi(\ul h,k)$ in the without-replacement model is
\begin{flalign*}
\Psi(\ul h,k)=&
\prod_{x \ne 0}
\brac{2\brac{ \sum_{\{u,v\} \ne \{x,0\} \atop u+v=x}\frac {h_\uu h_\uv}{(n-1)_2}
+ \frac{(h_\ux-1)h_0}{(n-1)_2} }
}^{h_\ux}
\brac{\frac{(h_0-1)_2}{(n-1)_2}+ \sum_{u \ne 0} \bfrac{(h_\uu)_2}{(n-1)_2}}^{h_0}\\
=&\Phi(\ul h, k)\bfrac{n^2}{(n-1)_2}^n \prod_{x \ne 0}
\brac{1-\frac{h_0}{\sum h_\uu h_\uv}}^{h_\ux}
\brac{1-\frac{\sum_\uu h_\uu+2h_0-2}{\sum h_\uu^2}}^{h_0}\\
=&\Phi(\ul h, k) \cdot C.
\end{flalign*}
As $h_i=(1+o(1))n/2^k$, and $(1-h_0/\sum h_\uu h_\uv)^{h_\ux}\sim e^{-2/2^k}$
we have
\[
\prod_{x \ne 0} \brac{1-\frac{h_0}{\sum h_\uu h_\uv}}^{h_\ux}\sim (e^{-2/2^k})^{2^k-1}
=e^{-2 +1/2^{k-1}},
\]
and
\[
\brac{1-\frac{\sum_\uu h_\uu+2h_0-2}{\sum h_\uu^2}}^{h_0}\sim \brac{1-\frac{2^k+2}{n}}^{n/2^k}=e^{-1-1/2^{k-1}}.
\]
Combining this with \eqref{e3} gives
\[
C \sim e^{3}e^{-2 +1/2^{k-1}}e^{-1-1/2^{k-1}} =1.
\]
\subsection{Without replacement} \label{SNRep}
Let $ S=\{2 \le \ell \le \om\}$ where $\om \rai$  slowly with $n$. For $\ell \in S$, let $Y_\ell(M)$ be the number of index sets of zero-sum rows of size $\ell$ in $M$. Similarly to \eqref{EXL}
\begin{flalign}
\E Y_\ell = & {n \choose \ell} \brac{2 \frac{(\ell-1)(n-\ell)}{(n-1)_2}}^\ell \brac{\bfrac{(\ell)_2}{(n-1)_2} + \bfrac{(n-1-\ell)_2}{(n-1)_2}}^{n-\ell} \label{EXLW}.
\end{flalign}
Assuming that $\ell=o(\sqrt n)$ then
\[
\E Y_\ell = \frac{(2(\ell-1))^\ell}{\ell!}e^{-2\ell} \ooi.
\]
If $L$ is zero-sum then the sub-matrix $M_{L,L}$ is the incidence matrix of a random functional digraph $D_L$ with no fixed points, in which case there are $\ell-1$ off-diagonal entries in any column of $M_{L,L}$
and we exclude cycles of size one. The probability that the underlying graph of
$D_L$ is connected is
\[
\Pr(D_L \text{ connected})= \frac{(\ell-1)!}{(\ell-1)^\ell} \sum_{j=0}^{\ell-2} \frac{\ell^j}{j!}.
\]

\end{document}